\documentclass[a4paper,twoside,11pt]{article}
\usepackage[utf8]{inputenc}
\usepackage{graphicx,epstopdf}
\usepackage{amsmath}
\usepackage{amsthm}
\usepackage{amssymb}
\usepackage{color}
\usepackage{mathtools}
\usepackage{subfig}
\usepackage{siunitx}
\usepackage{amsthm}
\usepackage[font=footnotesize,labelfont=bf]{caption}
\usepackage[author={Michele-Daniele-Anna-Michele}]{pdfcomment}

\usepackage{geometry,calc}
\DeclareMathOperator*{\tr}{\operatorname{tr}}

\newtheorem{theorem}{Theorem}[section]

\newtheorem{lemma}[theorem]{Lemma}
\theoremstyle{remark}
\newtheorem{remark}{Remark}

\usepackage[most]{tcolorbox}
\definecolor{darkgreen}{rgb}{0.0, 0.55, 0.0}
\let\div\relax
\DeclareMathOperator{\div}{\nabla\cdot}
\DeclareMathOperator{\bdiv}{\boldsymbol{\nabla}\cdot}
\newcommand{\rmdiv}{\rm{div}}
\newcommand{\rmbdiv}{\textbf{div}}

\renewcommand{\O}{\Omega}

\newcommand\Th{{\mathcal T}_h}

\newcommand\R{\mathbb{R}}

\def\decapita#1{}
\def\grecomultibold#1#2{\grecobolddef#1\def\secondobold{#2}%
	\ifx#2\finemultibold\let\next\relax\let\secondobold\relax
	\else\let\next\grecomultibold
	\fi\expandafter\next\secondobold}

\def\grecobolddef#1{%
	\edef\dadef{bf\expandafter\decapita\string#1}%
	\expandafter\def\csname\dadef\endcsname{{\neretto #1}}}

\def\neretto#1{\setbox0=\hbox{\mathsurround=0pt$#1$}%
	\kern.02em\copy0 \kern-\wd0
	\kern-.02em\copy0 \kern-\wd0
	\raise.03em\box0 \kern.02em}
\grecomultibold\alpha\phi\beta\omega\gamma\delta\xi\sigma\tau\eta\theta\varepsilon\rho\psi\pi\varphi\varepsilon\lambda\mu\finemultibold

\def\tr{{\rm tr}}

\def\wbox#1;#2;{\vbox{\hrule\hbox{\vrule height#1mm\kern#2mm\vrule
			height#1mm}\hrule}}

\let\phi\varphi

\def\bfzero{{\bf 0}}
\def\xt{{\tilde x}}
\def\yt{{\tilde y}}

\newcommand{\dEl}   {\text{d} E}
\newcommand{\df}    {\text{d} f}
\newcommand{\ds}    {\text{d} s}

\newcommand{\norm}[2][]{\left\lVert{#2}\right\rVert_{#1}}
\newcommand{\seminorm}[2][]{\left\lvert{#2}\right\rvert_{#1}}
\newcommand{\Poly}[3][]{\mathbb{P}^{#1}_{#2}\!\left(#3\right)}
\newcommand{\bs}{\boldsymbol}

\newcommand{\roundPrecision}{2}
\newcommand{\bfb}   {\mathbf{b}}

\newcommand{\bff}   {\mathbf{f}}
\newcommand{\bfg}   {\mathbf{g}}
\newcommand{\bfh}   {\mathbf{h}}

\newcommand{\bfn}   {\mathbf{n}}

\newcommand{\bfp}   {\mathbf{p}}

\newcommand{\bfr}   {\mathbf{r}}

\newcommand{\bft}   {\mathbf{t}}
\newcommand{\bfu}   {\mathbf{u}}
\newcommand{\bfv}   {\mathbf{v}}
\newcommand{\bfw}   {\mathbf{w}}
\newcommand{\bfx}   {\mathbf{x}}

\newcommand{\bfz}   {\mathbf{z}}
\newcommand{\bfA}   {\mathbf{A}}
\newcommand{\bfB}   {\mathbf{B}}

\newcommand{\bfE}   {\mathbf{E}}

\newcommand{\bfH}   {\mathbf{H}}
\newcommand{\bfK}   {\mathbf{K}}
\newcommand{\bfI}   {\mathbf{I}}
\newcommand{\bfL}   {\mathbf{L}}
\newcommand{\bfM}   {\mathbf{M}}

\newcommand{\bfT}   {\mathbf{T}}
\newcommand{\bfU}   {\mathbf{U}}

\newcommand{\bfW}   {\mathbf{W}}

\newcommand{\bbH}   {\mathbb{H}}

\def\RM{\mathbf{RM}}
\def\bxt{{\tilde \bfx}}

\newcommand{\bfsigmah}{\bfsigma_h}
\newcommand{\bftauh}{\bftau_h}

\newcommand{\bfwh}{\bfw_h}
\newcommand{\bfzh}{\bfz_h}
\newcommand{\bfuh}{\bfu_h}
\newcommand{\bfvh}{\bfv_h}
\newcommand{\ph}{p_h}
\newcommand{\qh}{q_h}
\newcommand{\partialt}{\partial_t}
\newcommand{\onehalf}{\frac{1}{2}}
\DeclareMathOperator*{\bdev}{\operatorname{\boldsymbol{\rm dev}}}
\newcommand{\ddt}{\frac{\rm{d}}{\rm{dt}}}
\newcommand{\nablas}{\nabla_{\rm{s}}}

\newcommand{\bfsigmaI}{\bfsigma_{\mathcal{I}}}
\newcommand{\bfuI}{\bfu_{\mathcal{I}}}
\newcommand{\bfwI}{\bfw_{\mathcal{I}}}
\newcommand{\pI}{p_{\mathcal{I}}}

\newcommand{\errsigma}{e^{\bfsigma}}
\newcommand{\erru}{e^{\bfu}}
\newcommand{\errw}{e^{\bfw}}
\newcommand{\errp}{e^{p}}
\newcommand{\errpA}{e^{p}_h}
\newcommand{\errpI}{e^{p}_\mathcal{I}}
\newcommand{\errwA}{e^{\bfw}_h}
\newcommand{\errwI}{e^{\bfw}_\mathcal{I}}
\newcommand{\errsigmaA}{e^{\bfsigma}_h}
\newcommand{\errsigmaI}{e^{\bfsigma}_\mathcal{I}}
\newcommand{\erruA}{e^{\bfu}_h}
\newcommand{\erruI}{e^{\bfu}_\mathcal{I}}

\title{Fully--Mixed Virtual Element Method for the Biot Problem}
\author{Michele Botti\thanks{MOX, Department of Mathematics, Politecnico di Milano, 20133 Milano, Italy (michele.botti@polimi.it, anna.scotti@polimi.it, michele.visinoni@polimi.it)}, Daniele Prada\thanks{Istituto di Matematica Applicata e Tecnologie Informatiche ``Enrico Magenes'' del Consiglio Nazionale delle Ricerche, 27100 Pavia, Italy (daniele.prada@polimi.it)},\;
Anna Scotti\footnotemark[1],\;
Michele Visinoni\footnotemark[1]}

\date{}

\begin{document}
%%%%%%%%%%%%%%%%%%%%%%%%%%%%%%%%%%%%%
\maketitle
	
\begin{abstract}
\noindent
Poroelasticity describes the interaction of deformation and fluid flow in saturated porous media. A fully-mixed formulation of Biot's poroelasticity problem has the advantage of producing a better approximation of the Darcy velocity and stress field, as well as satisfying local mass and momentum conservation. 
In this work, we focus on a novel four-fields Virtual Element discretization of Biot's equations. The stress symmetry is strongly imposed in the definition of the discrete space, thus avoiding the use of an additional Lagrange multiplier. 
A complete a priori analysis is performed, showing the robustness of the proposed numerical method with respect to limiting material properties. The first order convergence of the lowest-order fully-discrete numerical method, which is obtained by coupling the spatial approximation with the backward Euler time-advancing scheme, is confirmed by a complete 3D numerical validation. A well known poroelasticity benchmark is also considered to assess the robustness properties and computational performance.

\medskip\noindent
\textbf{AMS subject classification}: 65M12; 65M60; 74F10; 76S05.
		
\medskip\noindent
\textbf{Keywords}: Virtual Element method, poromechanics, Biot problem, mixed formulation, polyhedral meshes.
\end{abstract}

%%%%%%%%%%%%%%%%%%%%%%%%%%%%%
\section{Introduction}
\label{sec:intro}
%%%%%%%%%%%%%%%%%%%%%%%%%%%%%
The equations of linear poroelasticity describe the interaction of elastic deformation and fluid flow in fully saturated porous media and find application in many and diverse fields, ranging from geomechanics to the study of biological tissue and industrial products~\cite{cheng2016}. The equations describing fluid flow are the mass conservation equation and the Darcy law, suitably modified to account for the deformation of the solid skeleton and the compressibility of the solid and the fluid phases. The deformation of the porous medium is described by the usual laws of linear elasticity, however, deformation is linked to the effective stress, accounting for fluid pressure. The resulting two-ways coupled system is known as the Biot problem which, despite being linear, can pose some challenges in its numerical solution depending on the physical coefficients~\cite{kreuzer2024}.

The Biot problem has been approximated with methods such as the Finite Element Method (FEM), the Finite Volume Method (FVM), or traditionally a combination of the two, since FVM is usually applied to the Darcy problem, whereas standard FEM is a common choice for elasticity. The cheapest version of the problem is the two-fields formulation, the two unknowns being fluid pressure and solid displacement. 
Instead, in the so-called three-fields formulation, the flow problem is solved in its mixed form, i.e. the Darcy velocity is treated as an unknown and approximated with a suitable space. This formulation has been studied in the scope of FEM~\cite{Oyarzua2016, HU2017143}: in particular, several studies have focused on the choice of the spaces to obtain a cheap and inf-sup stable approximation \cite{rodrigo}, and on the development of suitable preconditioners \cite{Gaspar-Lisbona-Oosterlee-Vabishchevich_2007, adler}.  
In the scope of the Virtual Element Method (VEM), three-fields formulations are discussed in~\cite{Burger2021} and \cite{Tang2021}, whereas a hybrid mimetic finite-difference and virtual element scheme is presented in~\cite{BORIO2021113917}. In this context, the flexibility of the VEM scheme is utilized to approximate displacements in discrete mechanical problems. VEM has also been coupled with FVM~\cite{Coulet-Faille-Girault-Guy-Nataf:2020}, and with a multiscale approach for heterogeneous domains~\cite{SREEKUMAR2021113543}. 

A mixed form of the flow problem, despite having more globally coupled degrees of freedom than a primal formulation, has the advantage of producing a better approximation of the fluid velocity and satisfies local mass conservation. It is far less common to consider a mixed formulation of the mechanical problem, even if a better approximation of the stress field could be beneficial in several applications~\cite{boffi2013, grytz2012}. If we consider the stress tensor as an unknown of the problem we have to enforce its symmetry, with different possible strategies. Symmetry can be enforced weakly by introducing an additional unknown in the system that plays the role of Lagrange multiplier, and thus obtaining the so-called five-fields formulation \cite{afw, lee, Ahmed2019, Zhao2022}. Clearly, this results in a large linear system due to the degrees of freedom associated with the stress, and, on top of that, the Lagrange multipliers. Alternatively, one could use a least-square formulation of Biot problem~\cite{Korsawe2005, Tchonkova2008}, or seek the stress directly in the space of symmetric tensors, obtaining a four-fields formulation of Biot problem \cite{Yi2014}. Unfortunately, with FEM these spaces are quite complicated and expensive to define \cite{Ahmed2019, aw}. Conversely, this can be done with less effort in the scope of VEM, see for instance \cite{Visinoni:2024}. An alternative four-fields formulation using the solid displacement, the fluid pressure, the fluid flux, and the total pressure as unknowns, has been proposed in the framework of FVM and mixed schemes~\cite{Kumar2020}.
  
The goal of this work is to propose and analyze a four-fields formulation of the Biot problem discretized in two and three-dimensions with VEM. The stress is sought for in a space of symmetric tensors and we focus on the lowest possible order. This choice is justified by the low regularity of material parameters in the applications of interest, and it minimizes the number of globally coupled degrees of freedom. The presented analysis does not need a uniformly positive storage coefficient, and the error estimates are robust for nearly incompressible
materials. Moreover, in contrast to \cite{lee,Ahmed2019,Yi2014}, the general case of mixed non-homogeneous boundary conditions is considered. 

The paper is structured as follows. In Section~\ref{sec:model} we recall the four-fields formulation of Biot's poroelasticity equations; in Section~\ref{sec:vem} we present the VEM method, with particular attention to the three dimensional case, and provide a theoretical analysis of the resulting semi-discrete formulation, proving stability and convergence error estimates; in Section~\ref{sec:numerical} we test the method on three dimensional test cases, providing numerical evidence of the theory and of the robustness of the method with respect to the material parameters. The implementation is based on Vem++~\cite{dassi2023vem},  a C++ library specifically developed for working with VEM discretisations.

%%%%%%%%%%%%%%%%%%%%%%%%%%%%%
\section{Model Problem}
\label{sec:model}
%%%%%%%%%%%%%%%%%%%%%%%%%%%%%
We consider Biot's poroelasticity equations, modeling the Darcean flow in a deformable porous medium saturated by a fluid. According to the  poroelasticity theory, the medium is modeled as a continuous superposition of solid and fluid phases. 
Let $\Omega\subset\mathbb{R}^d$, $d\in\{2,3\}$, denote a bounded connected polytopal domain with boundary $\partial\Omega$ and outward normal $\bfn$. For a given time interval $(t_0,t_F)$, 
volumetric load $\bfb:\Omega\times[t_0,t_F)\to\mathbb{R}^d$, 
fluid source $\psi:\Omega\times(t_0,t_F)\to\mathbb{R}$, and
initial fluid content $\eta_0:\Omega\to\mathbb{R}$,
%displacement $\bfu^0:\Omega\to\mathbb{R}^d$,
%and initial pore pressure $p^0:\Omega\to\mathbb{R}$,
the linear poroelasticity problem consists in finding a 
displacement field $\bfu:\Omega\times\lbrack t_0,t_F)\to\mathbb{R}^d$, stress tensor $\bfsigma:\Omega\times\lbrack t_0,t_F)\to\mathbb{R}^{d\times d}_s$, with $\mathbb{R}^{d\times d}_s$ denoting the space of symmetric $d\times d$ matrices, pore pressure $p:\Omega\times\lbrack t_0,t_F)\to\mathbb{R}$, and Darcy velocity $\bfw:\Omega\times (t_0,t_F)\to\mathbb{R}^d$ satisfying 
\begin{subequations}\label{eq:biot:strong}
  \begin{alignat}{2}
     \label{eq:biot:strong:stressdef}
     &\bfsigma = \mathcal{C}\nablas\bfu -\alpha p \bfI &\qquad&\text{in $\Omega\times \lbrack t_0,t_F)$},\\
     \label{eq:biot:strong:mechanics}
    &-\bdiv \bfsigma = \bfb &\qquad&\text{in $\Omega\times \lbrack t_0,t_F)$},\\
             \label{eq:biot:strong:darcy}
    &\bfw = -\bfK\nabla p&\qquad&\text{in $\Omega\times (t_0,t_F)$},\\
    \label{eq:biot:strong:flow}
    &\partial_t (s_0 p + \alpha\div\bfu) + \div\bfw= \psi &\qquad&\text{in $\Omega\times (t_0,t_F)$},\\
    \label{eq:biot:strong:initial}
    & s_0 p(t_0) + \alpha\div\bfu(t_0) = \eta_0 &\qquad&\text{in $\Omega$}.
  \end{alignat}
\end{subequations}
  
In the constitutive equation for stress \eqref{eq:biot:strong:stressdef}, $\nablas$ denotes the symmetric part of the gradient operator acting on vector-valued fields, $\bfI$ is the identity matrix in $\mathbb{R}^{d\times d}$, $\mathcal{C}:\Omega\to\mathbb{R}^{d^4}$ is the uniformly elliptic fourth-order tensor-valued function expressing the linear stress-strain law, and $\alpha>0$ is the Biot--Willis coefficient. Let the trace operator be defined such that ${\rm tr}(\bftau) \coloneq \sum_i^d \bftau_{ii}$.
For homogeneous isotropic materials, $\mathcal{C}$ can be expressed in terms of the Lam\'e coefficients $\mu:\Omega\to[\underline{\mu},\overline{\mu}]$, with $0<\underline{\mu}<\overline{\mu}$, 
and $\lambda:\Omega\to[0,+\infty)$ as
$$
    \mathcal{C}\bftau = 2\mu\ \bftau + \lambda\ {\rm tr}(\bftau)\bfI
    \quad\text{ for all $\bftau\in \mathbb{R}^{d\times d}_{\rm s}$}.
$$
%
%We can define the \emph{effective stress} $\boldsymbol{\sigma^\star}=\bfsigma + \alpha p \mathbf{I_d}$, such that $\boldsymbol{\sigma^\star} = \mathcal{C}\nablas \bfu$.
  
In the mass conservation equation~\eqref{eq:biot:strong:flow}, $\partial_t$ denotes the time derivative, and $s_0\ge 0$ is the constrained specific storage coefficient, which measures the amount of fluid that can be forced into the medium by pressure increments due to the compressibility of the structure. The case of a solid matrix with incompressible grains corresponds to the limit value $s_0=0$. The tensor $\bfK:\Omega\to\mathbb{R}^{d\times d}_{\rm s}$ is the uniformly elliptic permeability tensor, already divided by the fluid viscosity for simplicity. For strictly positive real numbers $0<\underline{K}\le\overline{K}$, $\bfK$ satisfies 
\begin{equation*}\label{eq:ass_diff}
    \text{ $\underline{K}|\bfxi|^2\le\bfK(\bfx)\bfxi\cdot\bfxi\le
    \overline{K}|\bfxi|^2$ for almost every $\bfx\in\Omega$ and all $\bfxi\in\mathbb{R}^d$.}
\end{equation*}
The coupling coefficient $\alpha$ is in the range $[0,1]$, where $\alpha=0$ corresponds to decoupling flow and deformation.

Note that by plugging \eqref{eq:biot:strong:stressdef} into \eqref{eq:biot:strong:mechanics} and \eqref{eq:biot:strong:darcy} into \eqref{eq:biot:strong:flow} we can obtain the classical two-fields formulation for the displacement $\bfu$ and pore pressure $p$ unknowns. However, we are interested in a mixed formulation, which can be written by introducing the fourth-order compliance tensor $\mathcal{A}:\Omega\to\mathbb{R}^{d^4}$ defined such that $\mathcal{A}(\mathcal{C}\bftau)=\mathcal{C}(\mathcal{A}\bftau)=\bftau$ for all $\bftau\in\mathbb{R}^{d\times d}$. Its expression in terms of the Lam\'e coefficients is given by
\begin{equation}\label{eq:compliance}
\mathcal{A}\bftau 
  = \frac{1}{2\mu}\left(\bftau- \frac{\lambda}{2\mu + d\lambda} {\rm tr}(\bftau)\bfI\right)
  = \frac{\boldsymbol{\rm dev}(\bftau)}{2\mu} + \frac{{\rm tr}(\bftau)}{d^2 \kappa} \bfI,
\end{equation}
where the deviatoric operator is defined by~$\boldsymbol{\rm dev} (\bftau) = \bftau - d^{-1} \tr(\bftau)\bfI$ and $\kappa= d^{-1}(2\mu+d\lambda)\ge 2d^{-1}\underline{\mu} >0$ denotes the bulk modulus of the porous medium.
Owing to \eqref{eq:biot:strong:stressdef} and the definition of $\mathcal{A}$ in \eqref{eq:compliance}, we have 
$$
\div \bfu = \frac{{\rm tr}(\bfsigma)+d\alpha p}{2\mu + d\lambda} 
= {\rm tr}(\mathcal{A}\bfsigma+\alpha p\mathcal{A}\bfI).
$$
By plugging the previous expression into \eqref{eq:biot:strong:flow} and \eqref{eq:biot:strong:initial}, we rewrite \eqref{eq:biot:strong} as:
\begin{subequations}\label{eq:biot:mixed}
  \begin{alignat}{2}
    \label{eq:biot:mixed_const}
        &\mathcal{A}( \bfsigma + \alpha p \bfI) - \nablas \bfu =  \mathbf{0}  &   &\text{in $\Omega\times \lbrack t_0,t_F)$}, \\
        \label{eq:biot:mixed_bal}
        &- \bdiv \bfsigma =\bfb &   &\text{in $\Omega\times\lbrack t_0,t_F)$},\\
        %\label{eq:biot:mixed_sym}
        %&\bfsigma = \bfsigma^\top &   &\text{in $\Omega\times(t_0,t_F)$},\\
        \label{eq:biot:mixed_darcy}
        &\bfK^{-1}\bfw + \nabla p = \mathbf{0} &   &\text{in $\Omega\times(t_0,t_F)$},\\
        \label{eq:biot:mixed_massb}
        &\partial_t \left(s_0 p + \alpha{\rm tr}(\mathcal{A}(\bfsigma+\alpha p\bfI))\right) +
        \div \bfw = \psi &  \quad &\text{in $\Omega\times(t_0,t_F)$},\\
        \label{eq:biot:mixed:initial}
        & (s_0 +\alpha^2\kappa^{-1})\ p(t_0) + \alpha{\rm tr}(\mathcal{A}\bfsigma(t_0)) = \eta_0 &\qquad&\text{in $\Omega$}.
    \end{alignat}
\end{subequations}

We close the problem by prescribing general boundary conditions. We introduce two partitions $\partial \Omega = \partial_s \Omega \cup \partial_u \Omega$ and $\partial \Omega = \partial_w \Omega\cup \partial_p \Omega$. For simplicity, we assume that the Hausdorff measures of $\partial_u\Omega$ and $\partial_p\Omega$ are strictly positive, i.e., $|\partial_u\Omega|>0$ and $|\partial_p\Omega|>0$ and impose a given traction on $\partial_s\Omega$, a fixed displacement on $\partial_u\Omega$, a given flux on $\partial_w\Omega$, and a pressure on $\partial_p\Omega$, i.e.
\begin{gather}\label{eq:boundaryconds}
    \begin{aligned}
&\bfsigma(\bfx,t)\ \bfn(\bfx) = \bft(\bfx,t) & &\text{on }\partial_s\Omega\times\lbrack t_0,t_F)\\
&\bfu(\bfx,t) = \bfg_u(\bfx,t) & &\text{on }\partial_u\Omega\times\lbrack t_0,t_F)\\
&\bfw(\bfx,t)\cdot\bfn(\bfx) = f(\bfx,t) & &\text{on }\partial_q\Omega\times (t_0,t_F)\\
&p(\bfx,t) = g_p(\bfx,t) & &\text{on }\partial_p\Omega\times(t_0,t_F).
%&p(\bfx, 0) = p_0 & &\text{in }\Omega\\
%&\bfu(\bfx, 0) = \bfu_0 & &\text{in }\Omega
\end{aligned}
\end{gather}

In the mixed formulation of elasticity and poroelasticity problems, the symmetry of the stress tensor $\bfsigma$ has to be explicitly enforced, unlike the primal case where it follows from the presence of the strain operator, i.e. the symmetric gradient.
This constraint can be imposed in a weak sense by means of a Lagrange multiplier, leading to the so called five-fields formulation studied in \cite{lee}.  However, we decide to embed the symmetry constraint directly in the VEM discrete spaces avoiding this extra equation, as done in \cite{Visinoni:2024,Artioli-DeMiranda-Lovadina-Patruno:2017,Artioli-DeMiranda-Lovadina-Patruno:2018, Dassi-Lovadina-Visinoni:2020}.

%%%%%%%%%%%%%%%%%%%%%%%%%%%%%%%%%
\subsection{Weak formulation}
%%%%%%%%%%%%%%%%%%%%%%%%%%%%%%%%%%%%%%%
First, we introduce some notation. For $X\subseteq\Omega$, the notation $\bfL^2(X)$ is adopted in place of $[L^2(X)]^d$ and $\mathbb L^2(X)$ in place of $[L^2(X)]^{d\times d}$. The scalar product in $L^2(X)$ is denoted by $(\cdot,\cdot)_X$, with associated norm $\|\cdot\|_X$.  
Similarly, the Sobolev spaces $\bfH^s(X)$ are defined as $[H^s(X)]^d$, with $s> 0$, equipped with the norm $\|\cdot\|_{s,X}$. 
In addition, we will use $\bfH(\mathrm{div},X)$ to denote the space of $\bfL^2(X)$ functions with square integrable divergence.
Spaces of tensor fields defined over any $X\subset\overline{\Omega}$ are denoted by special Roman capitals and the subscript $\rm{s}$ is appended to denote the subspace of symmetric tensor fields. 
For example, $\mathbb{L}^2_s(X)$ is the spaces of square-integrable symmetric tensor fields and $\mathbb{H}_s(\rmbdiv,X)$ is the subspace of $\mathbb{L}^2_s(X)$ spanned by tensor fields having rows in $\bfH(\mathrm{div},X)$.

For an integer $m\ge 0$ and a vector space $V$ with scalar product $(\cdot,\cdot)_V$, the space $C^{m}(V)\coloneq C^{m}([t_0,t_F];V)$ is spanned by $V$-valued functions that are $m$-times continuously differentiable in the time interval $[t_0,t_F]$. 
%The space $C^{m}(V)$ is a Banach space when equipped with the norm 
%$$
%\|\varphi\|_{C^{m}(V)}:= \max_{0\le i\le m}\, %\max_{t\in[t_0,t_F]} 
%\|\partial_t^{i} \varphi(t)\|_V.
%$$
Similarly, the Bochner space $L^2(V)= L^2((t_0,t_F);V)$ is spanned by square-integrable $V$-valued functions of the time interval $(t_0,t_F)$ and, for $s\ge 0$, the Hilbert spaces $H^{s}(V)\coloneq H^{s}((t_0,t_F);V)$ are equipped with the norm $\norm[H^s(V)]{{\cdot}}$ induced by the scalar product
$$
(\varphi,\psi)_{H^{s}(V)}:=\sum_{j=0}^s \int_{t_0}^{t_F} (\partial_t^{j}\varphi(t),\partial_t^{j}\psi(t))_V {\rm d} t
\qquad \forall\varphi,\psi \in H^{s}(V).
$$
In \eqref{problem:variationalFormulation}, the notation $\langle \phi,\xi \rangle_{\Gamma}$, with $\Gamma\subset\partial\Omega$, is used for the duality product between two functions $\phi\in H^{\frac12}(\Gamma)$ and $\xi\in H^{-\frac12}(\Gamma)$ and the functional spaces are defined as follows:
\begin{align*}
\bs{\Sigma} &:= \{\bftau \in \mathbb{H}_s(\rmbdiv,\Omega)\,:\,
\langle\bftau\ \bfn,\bfphi\rangle_{\partial\Omega}=0 \ \ \forall \bfphi\in \bfH^1_{0,\partial_u\Omega}(\Omega)\},\; 
&& \bfU :=\bfL^2(\Omega),\\
\bfW &:= \{\bfz \in \bfH(\mathrm{div},\Omega)\,:\,\langle\bfw\cdot \bfn, \xi\rangle_{\partial\Omega}=0 \ \ \forall \xi\in H^1_{0,\partial_p\Omega}(\Omega)\},\;
&& Q:= L^2(\Omega).
\end{align*}
We endow the spaces $\bfU$ and $Q$ with the usual $L^2$ norm, and the spaces $\bs{\Sigma}$ and $\bfW$ with the norms given by
\begin{equation*}
    \norm[\bs{\Sigma}]{\cdot}^2:=\norm[\Omega]{\cdot}^2 +  \norm[\Omega]{\bdiv (\cdot)}^2 \qquad  \text{and} \qquad
    \norm[\bfW]{\cdot}^2:=\norm[\Omega]{\cdot}^2 +  \norm[\Omega]{\div (\cdot)}^2.
\end{equation*}
For the stress, we have chosen a space of symmetric tensors; thus, there is no need to additionally enforce the symmetry in the problem formulation. 

For simplicity, we assume in what follows that $\bft = \bf 0$ and $f=0$ in \eqref{eq:boundaryconds}. The general case of non-homogeneous Neumann boundary conditions can be obtained by minor modifications introducing $\bfH({\rm div})$-regular lifting of boundary data as in \cite[Appendix A]{Visinoni:25}. 
We also assume the following regularity for the problem data: $\bfb(\cdot,t)\in \bfL^2(\Omega)$, $\psi(\cdot,t)\in L^2(\Omega)$, $\bfg_u(\cdot,t)\in \bfH^{\frac12}(\partial_u\Omega)$, $g_p(\cdot,t)\in H^{\frac12}(\partial_p\Omega)$ for all $t\in (t_0, t_F)$, and $\eta_0\in L^2(\Omega)$.
Starting from the initial boundary value problem \eqref{eq:biot:mixed}, we multiply equations \eqref{eq:biot:mixed_const}, \eqref{eq:biot:mixed_bal}, \eqref{eq:biot:mixed_darcy}, \eqref{eq:biot:mixed_massb} by suitable test functions and, after integration by parts, we obtain the following weak formulation in space: 
find the solution $(\bfsigma, \bfu,\bfw,p)$ such that, for any $t\in (t_0, t_F)$ one has $\bfsigma(t)\in\bs{\Sigma}$, $\bfu(t)\in \bfU$, $\bfw(t)\in \bfW$, and $p(t)\in Q$ satisfying
\begin{subequations}~\label{problem:variationalFormulation}
       \begin{alignat}{2}
\label{problem:variationalFormulation:stress}
            &(\mathcal{A}(\bfsigma + \alpha p \bfI), \bftau)_\Omega + (\bfu,
            \div \bftau)_\Omega  =
            \langle\bfg_u, \bftau\bfn\rangle_{\partial_u \Omega} &&\forall
            \bftau\in \bs{\Sigma}\\
            \label{problem:variationalFormulation:displacement}
            &-(\bdiv \bfsigma, \bfv)_\Omega = (\bfb,
            \bfv)_\Omega && \forall \bfv \in \bfU\\
            \label{problem:variationalFormulation:velocity}
            &(\bfK^{-1} \bfw, \bfz)_\Omega - (p, \div
            \bfz)_\Omega = \langle g_{p}, \bfz \cdot \bfn\rangle_{\partial_{p} \Omega}
            && \forall \bfz \in \bfW\\
            \label{problem:variationalFormulation:pression}
            &(s_0 \partial_t p, q)_\Omega + (\mathcal{A}\partial_t (\bfsigma + \alpha p \bfI), \alpha q \bfI)_\Omega + (\div
            \bfw, q)_\Omega = (\psi, q)_\Omega && \forall q\in Q,
\end{alignat}
together with the initial condition 
$$
\left((s_0 +\alpha^2 \kappa^{-1}) p(t_0) + \alpha{\rm tr}(\mathcal{A} \bfsigma(t_0)), q\right)_\Omega  = 
(\eta_0, q)_\Omega \qquad \forall q\in Q.
$$
\end{subequations}
\begin{remark}[Solutions at $t=t_0$]~\label{rem:variational.initial}
    Assuming that $\bfb\in C^0(\bfL^2(\Omega))$ and $\bfg_u\in \hspace{-0.5mm}C^0(\bfH^{\frac12}(\partial_u\Omega))$, it is possible to prescribe the mechanical equilibrium at the initial time $t=t_0$. Specifically, we define $(\bfsigma(t_0),\bfu(t_0),p(t_0))\in \bs{\Sigma}\times\bfU\times Q$ as the initial stress, displacement, and pressure fields solving the steady problem
$$
        \begin{aligned}
            &(\mathcal{A}(\bfsigma(t_0) + \alpha p(t_0) \bfI), \bftau)_\Omega +(\bfu(t_0),
            \bdiv \bftau)_\Omega \hspace{-0.5mm} = \hspace{-0.5mm}
            \langle\bfg_u(t_0), \bftau\bfn\rangle_{\partial_u \Omega} &&\forall
            \bftau\in \bs{\Sigma}\\
            &-(\bdiv \bfsigma(t_0), \bfv)_\Omega = (\bfb(t_0),
            \bfv)_\Omega && \forall \bfv \in \bfU\\
            &(s_0 p(t_0), q)_\Omega + (\mathcal{A}(\bfsigma(t_0) + \alpha p(t_0) \bfI), \alpha q \bfI)_\Omega = (\eta_0, q)_\Omega && \forall q\in Q,
        \end{aligned}
$$
    which corresponds to a well-posed generalized Stokes problem in mixed form. 
\end{remark}

%%%%%%%%%%%%%%%%%%%%%%%%%%%%%
\section{Virtual Element Discretization}
\label{sec:vem}
%%%%%%%%%%%%%%%%%%%%%%%%%%%%%
In this section, we briefly present the mesh assumptions and the local discrete Virtual Elements spaces. We then introduce the discrete bilinear forms and functionals and discuss their computability. Finally, we present the semi-discrete formulation, along with its stability analysis and convergence error. Even if the model has been presented in the general case, from this point forward we will consider the three-dimensional case.
\subsection{Mesh assumptions}
Let $\{\mathcal{T}_h\}_h$ be a sequence of partitions of $\Omega$ into general polyhedra. We denote the diameter of each element $E$ by $h_E$ and refer to the mesh size of $\mathcal{T}_h$ with $h := \max_{E \in \mathcal{T}_h} h_E$.
We suppose that for all $h$, each element $E$ in $\mathcal{T}_h$ is a contractible polyhedron that fulfills the following mesh assumptions, see~\cite{projectors}:
\begin{itemize}
	\item[A1.] $E$ is star-shaped with respect to a ball $B_E$ of radius $r_E\geq \gamma \, h_E$;
	\item[A2.] every face $f$ of $E$ is star-shaped with respect to a disk $B_f$ of radius $r_f\geq \gamma\, h_f\geq \gamma^2 \, h_E$, where $h_f$ is the diameter of $f$;
	\item[A3.] every edge $e$ of $E$ satisfies $h_e \geq \gamma^2 \, h_E$, where $h_e$ is the lengthxv  of $e$;
\end{itemize}
where $\gamma$ is a suitable uniform positive constant. We observe that the above hypotheses could be relaxed according to~\cite{BLRXX, BrennerSungSmallEdges, CaoChen2018}. Finally, given an element $E\in \Th$ with $n^E_f$ faces $f$, we denote by $\bfx_E$ the barycenter of $E$. For any face $f$, we denote by $\bfx_f$ the barycenter of $f$, by $\bfn_{f}$ the outward normal to $f$, and by $\bft_{f}$ an arbitrary vector tangent to the face.

Henceforth, following the approach in~\cite[Remark 2]{Artioli-DeMiranda-Lovadina-Patruno:2017}, we assume that the Lam\'e parameters $\lambda$ and $\mu$ and the permeability tensor $\bf{K}$ are piecewise constant over any mesh $\Th$. Indeed, since we are focusing on a lowest-order approximation, taking the mean value approximation of these coefficients over each element would not deteriorate the convergence order of the scheme.

%Moreover, as is usual in the VEM framework, the above hypotheses will be stressed in the numerical experiments considering meshes with small faces and edges.
%%%%%%%%%%%%%%%%%%%%%%%%%%%%%%%
\subsection{Discrete spaces}
%%%%%%%%%%%%%%%%%%%%%%%%%%%%%%%%
We present the discrete spaces for the proposed VE scheme. We start with the approximation of the displacement and stress fields based on~\cite{Dassi-Lovadina-Visinoni:2020} and then move to the discretization of the flow terms following~\cite{Brezzi-Falk-Marini:2014}.

%We introduce the useful polynomial spaces $\RM(E)$ and $\bfT_h(f)$.
%
\paragraph{Space $\RM(E)$}
It is the space of local infinitesimal rigid body motions:
\begin{equation}\label{space:localRM}
\RM(E):=\left\{ \bfr(\bfx) = \bfalpha + \bfomega \wedge \big(\bfx -\bfx_E\big)\ \ \text{ s.t. }\ \bfalpha,\,\bfomega \in\R^3 \right\},
\end{equation}
whose dimension is 6. 

\paragraph{Space $\bfT_h(f)$}
For each face $f\in \partial E$, we introduce
\begin{multline}\label{eq:face_approx}
	\bfT_h(f):=\left\{ \bfpsi(\bxt)=\,\bft_{f} + a\big[\bfn_{f}\wedge(\bfx(\bxt) -\bfx_f)\big]+ p_1(\bxt)\bfn_{f},\right.\\
 \left.\text{s.t.}\ a\in \R,\ p_1(\bxt)\in\Poly{1}{f} \right\},
\end{multline}
where $\bfx(\bxt)$ is the position vector of a point on $f$, determined by the two local coordinates $\bxt$.
The dimension of such a space is $6$, indeed:
\begin{itemize}
	\item The tangent vector $\bft_f$ is determined by a linear combination of given linearly independent tangential vectors $\bft_1$ and $\bft_2$, i.e. $\bft_f=b_1\bft_1+b_2\bft_2$.
	\item The rotation $a\big[\bfn_{f}\wedge(\bfx(\bxt) -\bfx_f)\big]$ is determined by a scalar value $a\in\R$.
	\item The polynomial $p_1(\bxt)\in\Poly{1}{f}$ is bivariate with respect to the local face coordinate system, so it is determined by three parameters:
	\begin{equation*}
		p_1(\bxt)=c_1+c_2\left(\frac{\xt-\xt_f}{h_f}\right)+c_3\left(\frac{\yt-\yt_f}{h_f}\right).
	\end{equation*}
\end{itemize}
\begin{figure}[ht]
    \centering
            \includegraphics[width=0.24\textwidth]{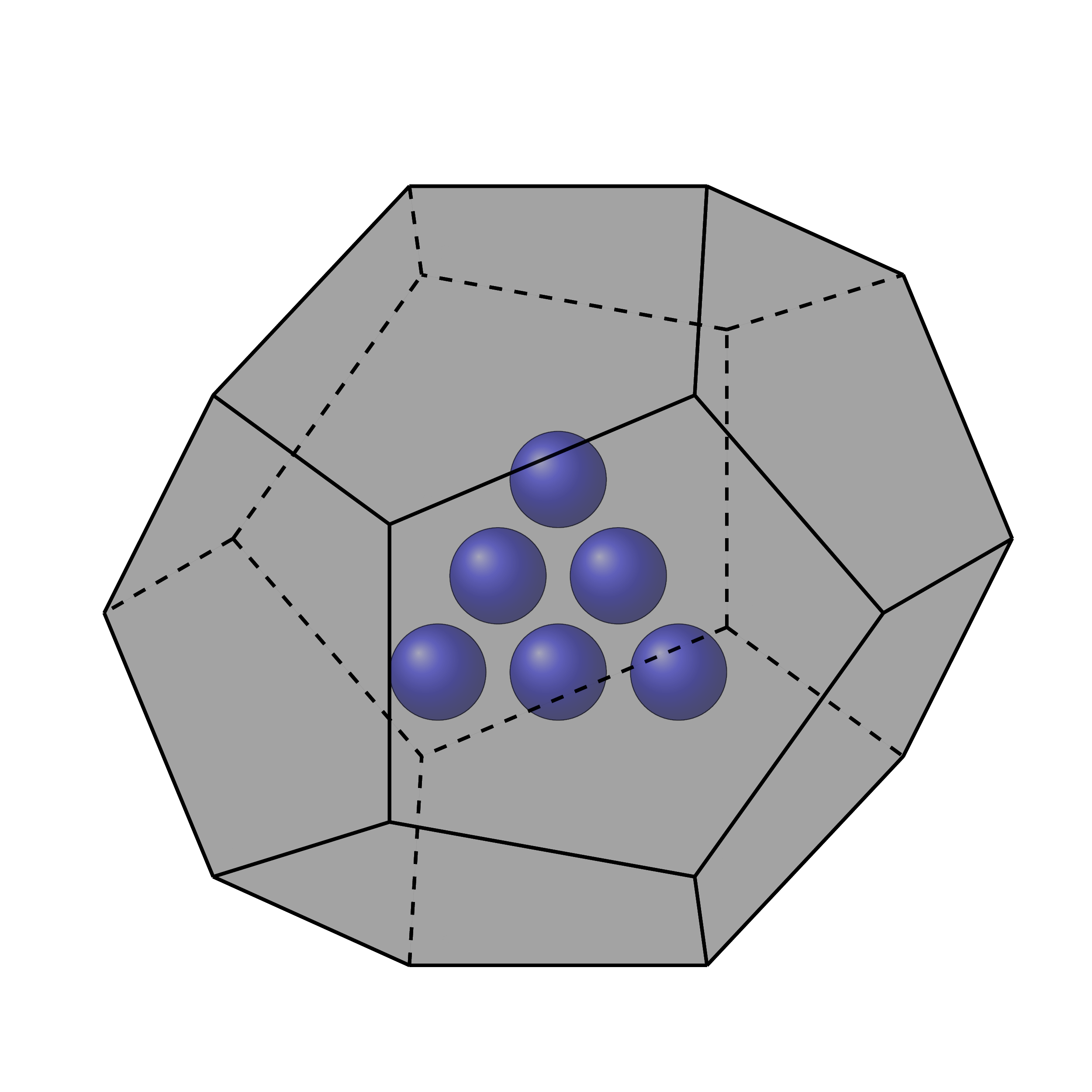}
        \includegraphics[width=0.24\textwidth]{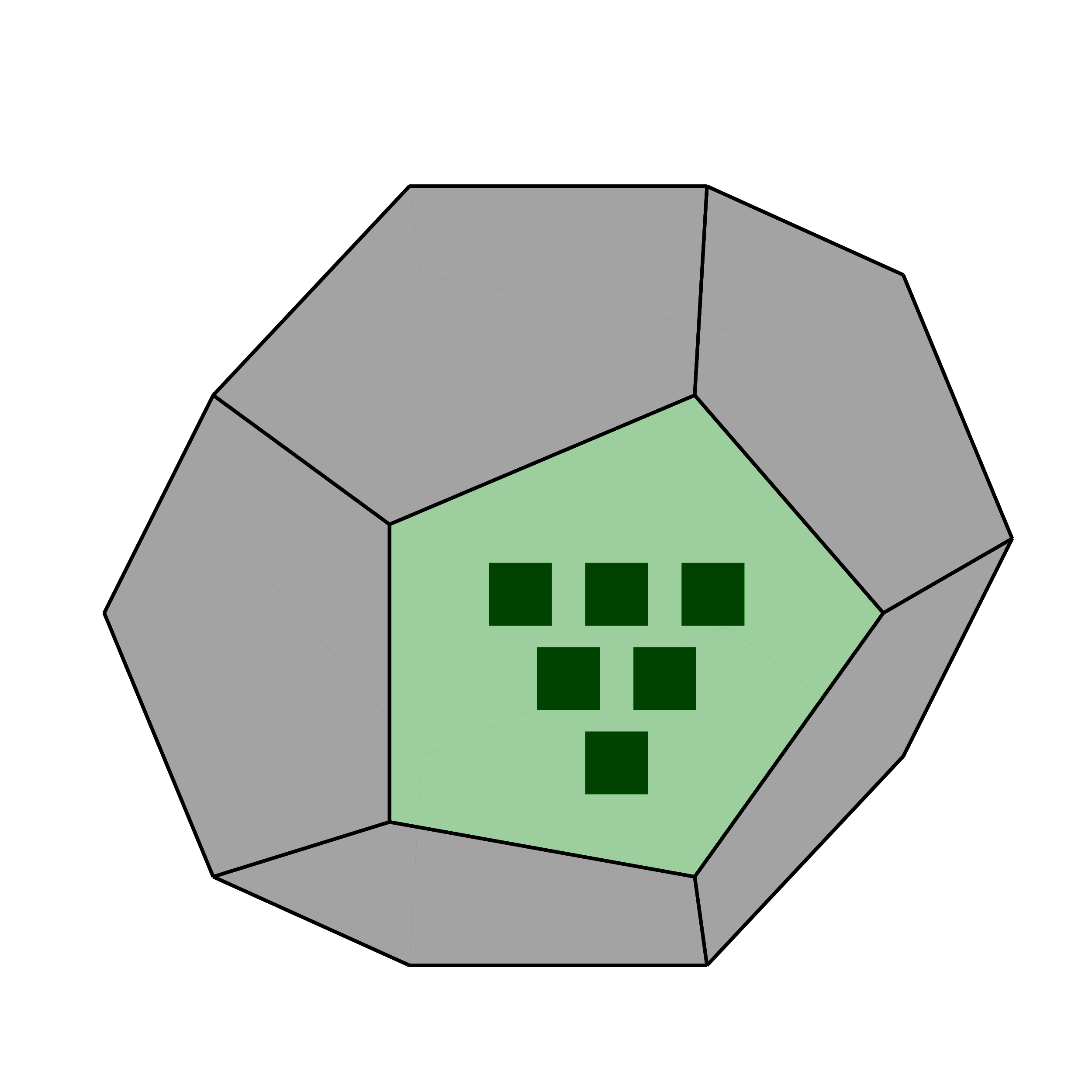}
            \includegraphics[width=0.24\textwidth]{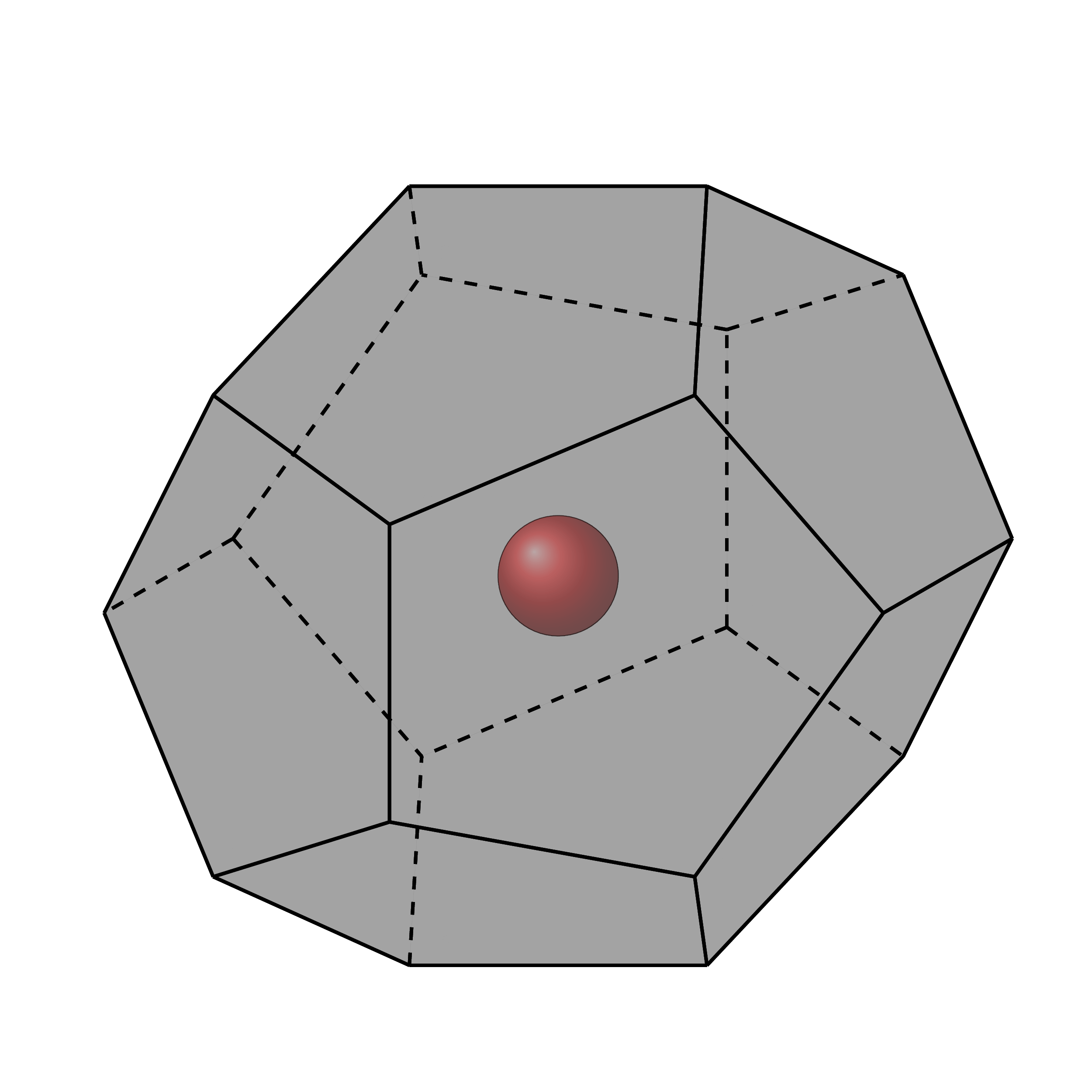}
        \includegraphics[width=0.24\textwidth]{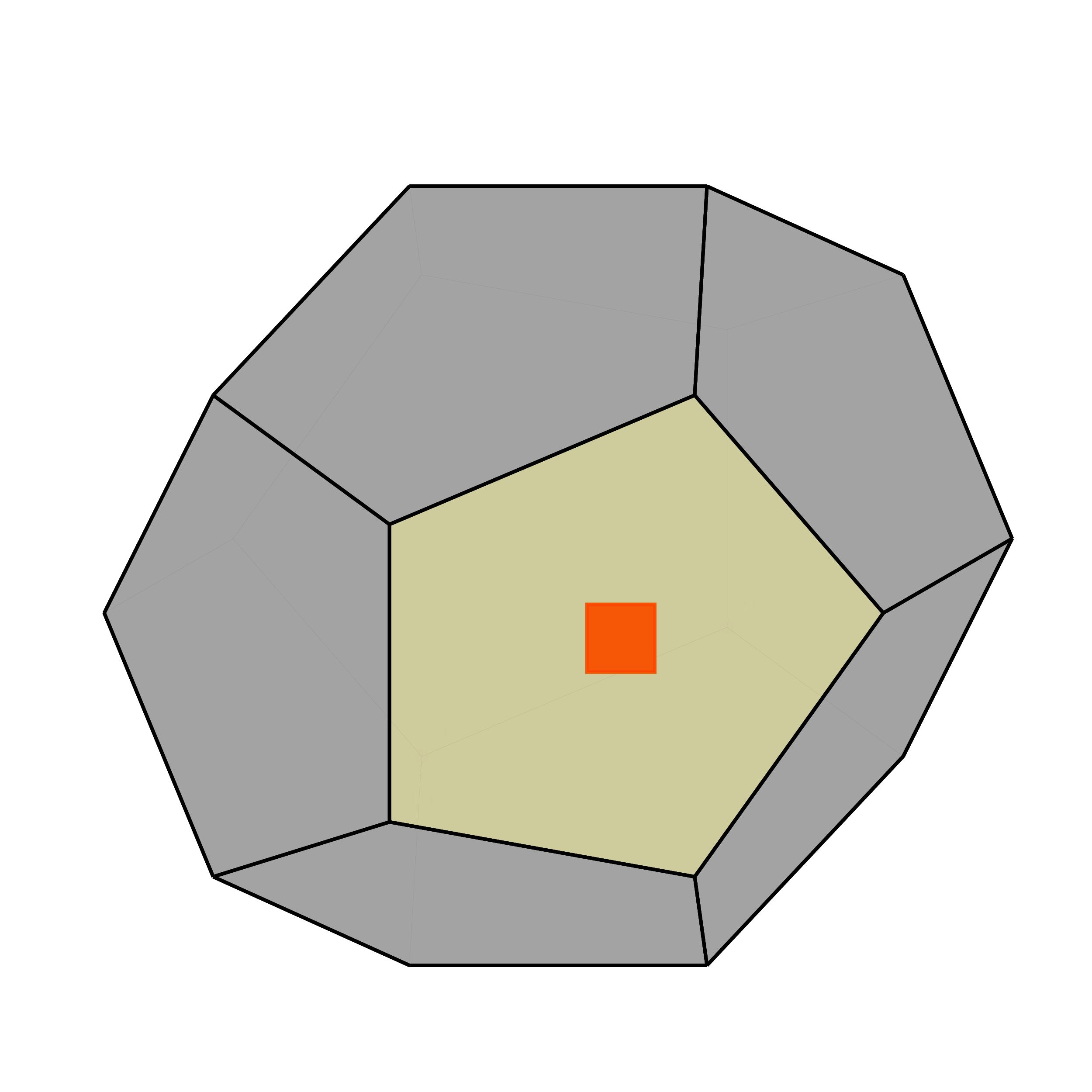}
        \caption{From left to right, the degrees of freedom are denoted as: blue sphere for displacement, green square for stress, red sphere for pressure, and orange square for velocity.}
    \label{fig:dofs}
\end{figure}
\paragraph{Stress space} We introduce our local approximation space for the stress field:
\begin{equation*} %\label{eq:local_stress_space}
	\begin{aligned}
		\bs{\Sigma}_h(E):=\Big\{ \bftau_h\in & \ \mathbb{H}_{\rm s}(\rmbdiv;E) : \exists\, \bfw^\ast\in \bfH^1(E) \mbox{ such that } \bftau_h=\mathcal{C}\nablas\bfw^\ast;\\ 
		&(\bftau_h\,\bfn)_{|f}\in \bfT_h(f) \quad \forall f\in \partial E;\quad \bdiv\bftau_h\in \RM(E) \Big\}. 
	\end{aligned}
\end{equation*}
Accordingly, for the local space $\bs{\Sigma}_h(E)$, for each face $f$ the following unisolvent set of degrees of freedom (DOFs) can be selected:

\begin{itemize}
	
	\item three DOFs fixing the tangential components of the traction on $f$:	
	\begin{equation}\label{stress_localdofs}
	%	\bftau_h \longrightarrow 
  \frac{1}{|f|}\int_f (\bftau_h\,\bfn)_{|f} \cdot \Big[ \bftheta_{f} + \alpha \big[\bfn_{f}\wedge(\bfx(\bxt) -\bfx_f)\big]\Big]~\df, 
	\end{equation} 
	where $\alpha\in\R$ and $\bftheta_f$ is an arbitrary vector tangential to $f$.
	
	\item three DOFs fixing the normal component of the traction on $f$:	
	\begin{equation}\label{stress_localdofs2}
	%	\bftau_h \longrightarrow 
  \frac{1}{|f|}\int_f (\bftau_h\,\bfn)_{|f} \cdot \left[q_1(\bxt)\bfn_{f}\right]~\df \qquad \forall\, q_1(\bxt)\in\Poly{1}{f}. 
	\end{equation} 
\end{itemize}
The divergence can be computed using the information on the boundary, see~\cite[Proposition 3.1]{Dassi-Lovadina-Visinoni:2020}, so the dimension of this space $6n_f^E$, see Figure~\ref{fig:dofs} for a graphical representation of the DOFs.
\paragraph{Displacement space}
For the displacement field, we introduce 
%the following polynomial space:
\begin{equation*}%\label{eq:local_displ}
%U_h(E)=\left\{ \bfv_h\in  \left[L^2(E)\right]^3\ :\ \bfv_h\in RM(E) \right\}.
\bfU_h(E)=\left\{ \bfv_h\in  \bfL^2(E)\ :\ \bfv_h\in \RM(E) \right\}.
\end{equation*}
Accordingly, for the local space $\bfU_h(E)$ the following six DOFs can be taken: %see~Figure~\ref{fig:dofs}:
\begin{equation}~\label{displacement_localdofs}
%\bfv_h \longrightarrow 
\frac{1}{|E|}\int_E \bfv_h \cdot \bfr~\dEl, \quad \forall \bfr \in \RM(E).
\end{equation}
%It follows that $\dim(U_h(E)) = 6$, see Figure~\ref{fig:dofs}.

For the flow variables, we consider the lowest order virtual Raviart-Thomas element introduced in~\cite[Remark 6.3, Remark 6.6]{Brezzi-Falk-Marini:2014}. 
\paragraph{Velocity space} We define the flow velocity space as follows:
\begin{equation*}
\begin{aligned}
\bfW_h(E):=\left\{ \bfz_h \in \bfH(\rmdiv;E)  :\ \exists\, \phi^\star\in H^1(E)
 \mbox{ such that } \bfz_h = \nabla \phi^\star;\right.\\ 
 \left.
 (\bfz_h\cdot\bfn)_{|f}\in\Poly{0}{f}
 \quad \forall f\in \partial E;\quad \nabla\cdot\bfz_h\in \Poly{0}{E} \right\}.
\end{aligned}
\end{equation*}
Accordingly, for the local space $\bfW_h(E)$ the following DOFs can be taken:
\begin{equation}~\label{velocity_localdofs}
%	\bfz_h \longrightarrow 
 \frac{1}{|f|}\int_f (\bfz_h\cdot \bfn)_{|f}~\df =(\bfz_h\cdot \bfn)_f.
\end{equation}
We remark that, once $(\bfz_h\cdot \bfn)_f = c_f\in\Poly{0}{E}$ is given for all $f\in\partial E$, the value of $\div \bfz_h\in\Poly{0}{E}$ is uniquely determined as follows
\begin{equation*}
\div \bfz_h = \frac{1}{|E|}\int_E \div\bfz_h~\dEl =  \frac{1}{|E|}\sum_{f\in\partial E}\int_f  (\bfz_h\cdot \bfn)_{|f}~\df =  \frac{1}{|E|}\sum_{f\in\partial E} |f| c_f.
\end{equation*}
Hence, the dimension of this space is  $n_f^E$, as depicted in Figure~\ref{fig:dofs}.
\paragraph{Pressure space} We define 
$
	Q_h(E) :=\left\{ p_h \in L^2(E) \ \colon \ p_h\in\Poly{0}{E} \right\}.
$
Thus, the single DOF for the pressure space can be taken as
\begin{equation}~\label{pressure_localdofs}
% p_h \longrightarrow 
 \frac{1}{|E|} \int_E p_h~\dEl.
\end{equation}
%and the dimension of this space is 1, see Figure~\ref{fig:dofs}.

We define the global counterparts of the local virtual element spaces introduced above as follows:
\begin{equation*} 
	\begin{aligned}
		\bs{\Sigma}_h&:=\left\{ \bftau_h\in  \bs{\Sigma}\ \colon \ \bftau_h|_E\in \bs{\Sigma}_h(E) \ \ \forall E \in \Th \right\}, \\  
		\bfU_h&:=\left\{ \bfv_h\in  \bfU\ \colon \ \bfv_h|_E\in \bfU_h(E) \ \ \forall E \in \Th\right\},\\
		\bfW_h&:=\left\{ \bfz_h \in \bfW \ \colon \ \bfz_h|_E \in \bfW_h(E) \ \ \forall E \in \Th \right\},\\
		Q_h &:=\left\{ p_h \in Q \ \colon \ p_h|_E\in Q_h(E) \ \ \forall E \in \Th \right\}.
	\end{aligned}
\end{equation*}
%%
%% Local forms
%%
\subsection{The local forms}
In this section, we discuss the discretization of the bilinear and linear forms appearing in \eqref{problem:variationalFormulation}. These discrete forms are computed elementwise. 
We start with the mixed terms: for every $ \bfvh \in \bfU_h$ and $\bftauh \in \bs{\Sigma}_h$, the term 
\begin{equation*}
(\bfvh,\bdiv \bftauh)_{\O} =  \sum_{E\in\Th} (\bfvh,\bdiv \bftauh)_{E} 
\end{equation*}
is computable via the degrees of freedom, since both factors are polynomials (see~\eqref{stress_localdofs}, \eqref{stress_localdofs2}, and \eqref{displacement_localdofs}). Similarly, for each $\qh \in Q_h$ and $\bfzh \in \bfW_h$,  the mixed term related to the Darcy problem
\begin{equation}
(\qh,\nabla\cdot \bfzh)_{\O} = \sum_{E\in\Th} (\qh,\nabla\cdot \bfzh)_{E}
\end{equation}
is just computed using the degrees of freedom in~\eqref{velocity_localdofs} and~\eqref{pressure_localdofs}.
Moreover, recalling the definition of the compliance tensor in \eqref{eq:compliance}, for every $\ph, \qh \in Q_h$ the mass term can be rewritten as follows
%\[
%	(\mathcal{A} \ph \bfI, \qh \bfI)_{\O} = 
        %\sum_{E\in\Th}(\mathcal{A} \ph \bfI,  \qh \bfI)_{E}\\
%		\sum_{E\in\Th}(\ph \mathcal{A} \bfI ,\qh \bfI)_{E}
%		= \sum_{E\in\Th}\left( \frac{p\bfI}{3\lambda + 2 \mu} , q\bfI\right)_{E} \hspace{-0.5mm}
%		= \sum_{E\in\Th} \left(\frac{3p}{3\lambda + 2 \mu} ,q\right)_{E}.
%\]
\[
	(\alpha \ph \mathcal{A} \bfI, \alpha \qh \bfI)_{\O}  
        %=\sum_{E\in\Th}(\mathcal{A} p \bfI,  q \bfI)_{E}\\
		%\sum_{E\in\Th}(\alpha \ph \mathcal{A} \bfI , \alpha \qh \bfI)_{E}
		= \sum_{E\in\Th}\left(\frac{\alpha^2}{3\kappa} \ph\bfI, \qh \bfI\right)_{E}
		= \sum_{E\in\Th} \left( \alpha^2\kappa^{-1}\ph  ,\qh\right)_{E},
\]
where we recall that for $d=3$ we have $3\kappa = (2\mu+3\lambda)$. The previous term as well as
$
\left(s_0 \ph,\qh \right)_{\Omega}
$
are computable since the pressure is elementwise constant.

The elasticity bilinear form
$
(\mathcal{A}\bfsigmah, \bftauh)_{\O} =\sum_{E\in\Th} (\mathcal{A}\bfsigmah, \bftauh)_{E} 
$
for $\bfsigmah, \bftauh \in \bs{\Sigma}_h$
is not directly computable from the DOFs since both entries are virtual tensor fields. Following~\cite{Dassi-Lovadina-Visinoni:2020}, we introduce the approximation
\begin{equation*}%\label{eq:discrete_bilinearform:elasticity}
a_h^E(\bfsigmah, \bftauh) = (\mathcal{A}\,  \bs{\Pi}_{\rm s}^{E}\bfsigmah, \bs{\Pi}_{\rm s}^{E}\bftauh )_E + s_1^E((\bfI-\bs{\Pi}_{\rm s}^{E})\bfsigmah,(\bfI-\bs{\Pi}_{\rm s}^{E})\bftauh)
\quad\forall \bfsigmah, \bftauh \in \bs{\Sigma}_h,
\end{equation*}  
where the $L^2$-projection operator $\bs{\Pi}_{\rm s}^E:\bs{\Sigma}_h(E)\rightarrow \Poly{0}{E}^{3\times 3}_{\rm s}$ is given by 
\begin{equation}\label{eq:projection:elasticity}
	\int_E \bs{\Pi}^E_{\rm s} \bftauh : \bfpi_0~\dEl = 	\int_E \bftauh : \bfpi_0~\dEl \qquad \forall \bfpi_0\in \Poly{0}{E}^{3\times 3}_{\rm s}
\end{equation}
As it was shown in~\cite{Dassi-Lovadina-Visinoni:2020}, the DOFs~\eqref{stress_localdofs} and~\eqref{stress_localdofs2} allow the explicit computation of the projection $\bs{\Pi}_{\rm s}^E$. Indeed, we notice that each $\bfpi_0\in \Poly{0}{E}^{3\times 3}_{\rm s}$ can be written as the symmetric gradient of a $\bfp_1\in \Poly{1}{E}^{3}$, i.e. $\bfpi_0 = \nablas \bfp_1$.  Hence, using the divergence theorem, the right-hand side of~\eqref{eq:projection:elasticity} becomes
\begin{equation}\label{eq:projection:elasticity_rhs}
\int_E \bftauh : \bfpi_0~\dEl 
%= \int_E \bftauh : \nablas \bfp_1~\dEl 
= -\int_E (\bdiv \bftauh) \cdot \bfp_1~\dEl
+ \int_{\partial E} \bftauh\ \bfn \cdot \bfp_1~\df
\end{equation}
which is computable from the DOFs. The stabilization term is defined as
\begin{equation}\label{eq:stabilization:elasticity}
s_1^E(\bfsigmah,\bftauh) = \xi_{1,E} h_E \sum_{f\subset\partial E} \int_f (\bfsigmah\,\bfn)\cdot( \bftauh\,\bfn)~\df\qquad\forall \bfsigmah, \bftauh \in \bs{\Sigma}_h,
\end{equation}
where $\xi_{1,E}$ is a positive constant to be chosen according to tensor $\mathcal{A}$. In the numerical experiments, we choose $\xi_{1,E}=\frac{1}{2}\tr(\mathcal{A}_{|E})$.
% \mb{$\xi_{1,E}=\frac{1}{2} \max_{\bfx\in E}\tr(\mathcal{A}(\bfx))$ requiring $\mathcal{A}$ to be elementwise regular. In the case the tensor field $\mathcal{A}$ is elementwise constant, we simply take $\xi_{1,E}=\frac{1}{2}\tr(\mathcal{A}_{|E})$. }\\

Similarly, the Darcy bilinear form
$(\bfK^{-1} \bfwh,\bfzh)_{\O}  =
\sum_{E\in\Th} (\bfK^{-1} \bfwh,\bfzh)_{E}$
for $\bfwh,\bfzh \in \bfW_h$ is not computable, thus we define the local form 
\begin{equation*}
	c_h^E(\bfwh,\bfzh) = (\bfK^{-1} \bs{\Pi}^{E}\bfwh, \bs{\Pi}^{E}\bfzh)_{E} + s_2^E((\bfI-\bs{\Pi}^{E})\bfwh,(\bfI-\bs{\Pi}^{E})\bfzh). 
    %\quad\forall \bfwh,\bfzh \in \bfW_h.
\end{equation*}
Here, we use the $L^2$-projection operator  $\bs{\Pi}^{E}:\bfW_h(E)\rightarrow \Poly{0}{E}^3$ defined as
\begin{equation}\label{eq:projection.velocity}
\int_E \bs{\Pi}^{E}\bfzh \cdot \bfp_0~\dEl = \int_E \bfzh\cdot \bfp_0~\dEl \quad \quad \forall \bfp_0 \in  \Poly{0}{E}^3
\end{equation}
and stabilization bilinear form such that
\begin{equation*}
	s_2^E(\bfwh,\bfzh) := \xi_{2,E} h_E \sum_{f\subset \partial E} \int_f (\bfwh\cdot\bfn)(\bfzh\cdot\bfn)~\df,
\end{equation*}
with $\xi_{2,E}=\frac{1}{2}\tr(\bfK^{-1}_{|E})$. The computability of the projection in \eqref{eq:projection.velocity} can be inferred proceeding as in \eqref{eq:projection:elasticity_rhs}.
\begin{remark}~\label{remark:stability_stabterms}
For all $\bftauh\in\bs{\Sigma}_h(E)$, $\bfzh\in\bfW_h(E)$, the stabilizing terms $s_1(\cdot,\cdot)$ and $s_2(\cdot,\cdot)$ satisfy the following relations (see, e.g., \cite[Proposition 3.1]{Visinoni:25})
\begin{equation}~\label{eq:stability_estimates_stab}
    \begin{aligned}
        0\le \gamma_{1,*} \norm[E]{\mu^{-\onehalf}\bftauh}^2 \leq & \ s_1^E(\bftauh,\bftauh) \leq \gamma_1^* \norm[E]{\mu^{-\onehalf}\bftauh}^2
        \\
        0\le \gamma_{2,*} \norm[E]{\bfK^{-\onehalf}\bfzh}^2\leq & \ s_2^E(\bfzh,\bfzh) \leq \gamma_2^*\norm[E]{\bfK^{-\onehalf}\bfzh}^2
    \end{aligned}
\end{equation}
where the constants $\gamma_{1,*}, \gamma_1^*,\gamma_{2,*}$ and $\gamma_2^*$ depend only on the shape regularity. 
\end{remark}

Finally, for the coupling term, for each  $\qh \in Q_h$  and $\bftauh \in \bs{\Sigma}_h$ we have
\begin{equation}~\label{eq:coupling_term_discrete}
	(\alpha \qh\mathcal{A}  \bfI, \bftauh)_{\O} 
    %= \sum_{E\in\Th}(\mathcal{A} q \bfI, \bftau)_{E}
	= \sum_{E\in\Th} \left(\frac{\alpha}{3\kappa} \qh\bfI,\bs{\Pi}_{\rm s}^E\bftauh\right)_{E}
	= \sum_{E\in\Th} \left(\alpha \qh, \frac{\tr(\bs{\Pi}_{\rm s}^E\bftauh)}{3\kappa}\right)_{E}
\end{equation}
which is computable since $\qh$ is piecewise constant by using  \eqref{eq:projection:elasticity_rhs}.
The right-hand-side terms are computable (see~\eqref{stress_localdofs},~\eqref{stress_localdofs2},~\eqref{displacement_localdofs},~\eqref{velocity_localdofs}, and~\eqref{pressure_localdofs}) by suitable quadrature rules on polytopal elements, see for instance~\cite{Mousavi-Sukumar:2011,Chin-Dassi-Manzini-Sukumar:2024,Sommariva-Vianello:2024}:
\begin{equation}
	\begin{aligned}
        & \langle \bfg_u, \bftauh\, \bfn\rangle_{\partial_u\Omega} = \sum_{f\subset\partial _u\Omega} (\bfg_u, \bftauh\, \bfn)_f
        &\quad &\forall \bftauh \in \bs{\Sigma}_h,\\
        &(\bfb, \bfvh)_{\O} =\sum_{E\in\Th} (\bfb, \bfvh)_{E} & \quad &\forall \bfvh \in \bfU_h, \\
        &\langle g_p, \bfzh\cdot \bfn\rangle_{\partial_p\Omega} =\sum_{f\subset\partial_p\Omega}  (g_p, \bfzh\cdot \bfn)_f
        &\quad &\forall \bfz \in \bfW_h,\\
		&(\psi, \qh)_{\O} =\sum_{E\in\Th} (\psi, \qh)_{E} & \quad &\forall \qh \in Q_h. 	
	\end{aligned}
\end{equation}
To conclude, in view of Remark~\ref{remark:stability_stabterms}, for all $\bfsigmah,\bftauh\in\bs{\Sigma}_h$ and $\bfwh,\bfzh\in\bfW_h$ the discrete bilinear forms
\begin{equation*}
    a_h(\bfsigmah,\bftauh) = \sum_{E\in\Th} a_h^E(\bfsigmah, \bftauh) \quad \mbox{and} \quad  c_h(\bfwh, \bfzh) = \sum_{E\in\Th} c_h^E(\bfwh, \bfzh),
\end{equation*}
are coercive and bounded in the sense that: for each $E\in\Th$
\begin{equation}\label{eq:coercivity_continuity_bilinearforms}
    \begin{aligned}
        &a_h^E(\bftauh,\bftauh) \geq \min\{1,\gamma_{1,*}\}\norm[E]{\mu^{-\onehalf}\bftauh}^2 = C_{s,*}\norm[E]{\mu^{-\onehalf}\bftauh}^2, \\
        &a_h^E(\bfsigmah,\bftauh) \leq \max \{1,\gamma_1^*\}\norm[E]{\mu^{-\onehalf}\bfsigmah}\norm[E]{\mu^{-\onehalf}\bftauh} %=  C_{s}^*\norm[E]{\mu^{-\onehalf}\bfsigmah}\norm[E]{\mu^{-\onehalf}\bftauh}
        , \\
          &c_h^E(\bfzh,\bfzh) \geq \min\{1,\gamma_{2,*}\}\norm[E]{\bfK^{-\onehalf}\bfzh}^2 = C_{v,*}\norm[E]{\bfK^{-\onehalf}\bfzh}^2, \\
        & c_h^E(\bfwh,\bfzh) \leq \max \{1,\gamma_2^*\} \norm[E]{\bfK^{-\onehalf}\bfwh} \norm[E]{\bfK^{-\onehalf}\bfzh}. %= C_{v}^*\hspace{-1mm}\norm[E]{\bfK^{-\onehalf}\bfwh}\norm[E]{\bfK^{-\onehalf}\bfzh}.
    \end{aligned}
\end{equation} 

\subsection{Global Semi--Discrete Formulation}

The semi-discrete virtual element formulation of Problem~\eqref{problem:variationalFormulation} reads as follows:  find $\bfsigmah \in H^1(\bs{\Sigma}_h)$, $ \bfuh\in H^1(\bfU_h)$, 
$\bfwh\in L^2(\bfW_h)$, $\ph \in H^1(Q_h)$ such that, for all $t\in (t_0,t_F]$ and $\bftauh \in \bs{\Sigma}_h$, $\bfvh \in \bfU_h$, $\bfzh \in \bfW_h$, $\qh \in Q_h$, it holds
\begin{subequations}\label{problem:semidiscrete_formulation}
\begin{alignat}{2}
\label{problem:semidiscrete_formulation:stress}
&a_h(\bfsigmah, \bftauh)
+
\left(
\bfuh,\bdiv \bftauh
\right)_\Omega 
+ 
\left(
\alpha\mathcal{A}\bfI \ph , \bftauh
\right)_\Omega 
= 
\langle 
\bfg_u, \bftauh\bfn
\rangle_{\partial_u \Omega} \\
\label{problem:semidiscrete_formulation:displacement}
&-\left(
\bdiv \bfsigmah, \bfvh
\right)_\Omega 
= (\bfb,\bfvh)_\Omega\\
\label{problem:semidiscrete_formulation:velocity}
&c_h(\bfwh,\bfzh)
- \left(\ph, \div\bfzh\right)_\Omega 
= 
\langle g_{p}, \bfzh \cdot \bfn\rangle_{\partial_{p} \Omega}\\
\label{problem:semidiscrete_formulation:pression}
&\left(\partialt \bfsigmah , \alpha \mathcal{A}\bfI\qh \right)_\Omega + (\div\bfwh, \qh)_\Omega +\left((s_0+\alpha^2\kappa^{-1})\partialt \ph, \qh\right)_\Omega = (\psi, \qh)_\Omega.
\end{alignat}
\end{subequations}
According to Remark \ref{rem:variational.initial}, assuming sufficient regularity, the initial solutions $\bfsigmah(t_0) \in \bs{\Sigma}_h$, $ \bfuh(t_0)\in \bfU_h$, and $\ph(t_0)\in Q_h$ can be computed such that
\begin{equation}\label{eq:discrete_initial}
\begin{aligned}
&\hspace{-1mm}a_h(\bfsigmah(t_0), \bftauh)
+
\left(
\bfuh(t_0),\bdiv \bftauh
\right)_\Omega 
+ 
\left(
\alpha\mathcal{A}\bfI \ph(t_0) , \bftauh
\right)_\Omega 
= 
\langle 
\bfg_u(t_0), \bftauh\bfn
\rangle_{\partial_u \Omega} \\
&\hspace{-1mm}-\left(
\bdiv \bfsigmah(t_0), \bfvh
\right)_\Omega 
= (\bfb(t_0),\bfvh)_\Omega\\
&\hspace{-1mm}\left(\bfsigmah(t_0) , \alpha \mathcal{A}\bfI\qh \right)_\Omega+\left((s_0+\alpha^2\kappa^{-1}) \ph(t_0), \qh\right)_\Omega = (\eta_0, \qh)_\Omega.
\end{aligned}
\end{equation}

Existence and uniqueness of solutions of the semidiscrete problem \eqref{problem:semidiscrete_formulation} can be established in the framework of differential algebraic equations as in \cite[Section III.A]{Yi2014}. The remaining part of the section is devoted to the derivation of suitable stability and error estimates. 
\begin{theorem}[Stability]~\label{theorem:stability}
Let $(\bfsigmah(t),\bfuh(t),\bfwh(t),\ph(t))$
solve~\eqref{problem:semidiscrete_formulation} for each $t\in(t_0,t_f]$ and assume the problem data to be regular enough to define 
\begin{equation}\label{eq:stab.rhs}
\begin{aligned}
\mathcal{C}_0(\bfb, \bfg_u, g_p, \eta_0) &= 
\norm[\Omega]{\bfb(t_0)}^2 +
\norm[\bfH^{\frac12}(\partial_u\Omega)]{\bfg_u(t_0)}^2 +
\norm[H^{\frac12}(\partial_p\Omega)]{g_p(t_0)}^2 + 
\norm[\Omega]{\eta_0}^2,
\\
\mathcal{C}(\bfb, \psi, \bfg_u, g_p) &= 
\hspace{-1mm}\norm[H^1(\bfL^2(\Omega))]{\bfb}^2 \hspace{-1mm}+ \hspace{-1mm}\norm[L^2(L^2(\Omega))]{\psi}^2 \hspace{-1mm}+\hspace{-1mm}
\norm[H^1(\bfH^{\frac12}(\partial_u\Omega))]{\bfg_u}^2 \hspace{-1mm}+\hspace{-1mm}
\norm[H^1(H^{\frac12}(\partial_p\Omega))]{g_p}^2\hspace{-1mm}.
\end{aligned}
\end{equation}
Then, there exists a constant $C>0$ independent of $s_0$, $\lambda$, the length of the time interval $(t_f-t_0)$, and the mesh size $h$, such that
\begin{equation}~\label{eq:stability_theorem}
\begin{aligned}
&\sup_{t\in[t_0,t_f]}\norm[\Omega]{\bfuh(t)}^2 +
\sup_{t\in[t_0,t_f]}\norm[\boldsymbol{\Sigma}]{(2\mu)^{-\frac12}\bfsigmah(t)}^2 +
\int_{t_0}^{t_f}\norm[\Omega]{\bs{K}^{-\frac12}\bfwh(s)}^2\ds +\\
&\qquad
\int_{t_0}^{t_f}\norm[\Omega]{\ph(s)}^2 \ds\leq 
C\big((t_f-t_0)\ \mathcal{C}(\bfb, \psi, \bfg_u, g_p) +
\mathcal{C}_0(\bfb, \bfg_u, g_p, \eta_0)\big).     
\end{aligned}
\end{equation}
The estimate remains valid in the incompressible limits $s_0\rightarrow 0$ and $\lambda\rightarrow\infty$.
\end{theorem}
\begin{proof}
%For simplicity, and without loss of generality, we assume that the diameter of our domain is approximately equal to $1$ \todo{cosa succede se non \`e cosi`?}, i.e., $h_\Omega \sim 1$. 
The proof is divided into five steps.

\paragraph{Step 1: energy estimate} 
We differentiate~\eqref{problem:semidiscrete_formulation:stress} with respect to time and we choose as a test function $\bftauh = \bfsigmah$. We get 
\begin{equation}\label{eq:startstab}
a_h(\partialt \bfsigmah, \bfsigmah) +  (\partialt \bfuh,\bdiv \bfsigmah)_\Omega +  \left(\alpha  \mathcal{A}\bfI\partialt \ph, \bfsigmah\right)_\Omega 
= \langle \partialt\bfg_u, \bfsigmah\bfn\rangle_{\partial_u \Omega}.
\end{equation}
Then, we take $\bfvh=\partialt \bfuh$ in~\eqref{problem:semidiscrete_formulation:displacement}, $\bfzh=\bfwh$ in~\eqref{problem:semidiscrete_formulation:velocity}, $\qh=\ph$ in~\eqref{problem:semidiscrete_formulation:pression} and gathering the obtained results together with \eqref{eq:startstab}, we obtain
\begin{equation*}
    \begin{aligned}
        &a_h(\partialt\bfsigmah,\bfsigmah) +
\left(
\alpha\mathcal{A}\bfI \partialt \ph, \bfsigmah
\right)_\Omega +
\left(
\alpha   \mathcal{A}\bfI\ph, \partialt \bfsigmah
\right)_\Omega +
\left(
(s_0 + \alpha^2 \kappa^{-1} )
\partialt \ph, \ph
\right)_\Omega + 
\\ &\quad c_h(\bfwh, \bfwh) = \langle 
\partialt\bfg_u, \bfsigmah\bfn
\rangle_{\partial_u \Omega} +
\left(
\bfb,\partialt \bfuh
\right)_{\Omega} +
\langle g_{p}, \bfwh \cdot \bfn\rangle_{\partial_{p} \Omega} + 
(\psi, \ph)_\Omega.
    \end{aligned}
\end{equation*}
For the sake of presentation, we write the previous identity as 
\begin{equation}\label{eq:T_b_RHS}
\mathcal{N} + c_h(\bfwh, \bfwh) = \mathcal{R},
\end{equation}
where 
\begin{equation*}
\begin{aligned}
\mathcal{N}\hspace{-0.5mm} &=\hspace{-0.5mm}
a_h(\partialt\bfsigmah,\bfsigmah) \hspace{-1mm}+\hspace{-1mm}
\left(
\alpha\mathcal{A}\bfI \partialt \ph, \bfsigmah
\right)_\Omega \hspace{-1mm}+\hspace{-1mm}
\left(
\alpha   \mathcal{A}\bfI\ph, \partialt \bfsigmah
\right)_\Omega \hspace{-1mm}+\hspace{-1mm}
\left(
\frac{s_0\kappa +\alpha^2}\kappa
\partialt \ph, \ph
\right)_\Omega,
\\
   \mathcal{R} &= \langle 
\partialt\bfg_u, \bfsigmah\bfn
\rangle_{\partial_u \Omega} +
\left(
\bfb,\partialt \bfuh
\right)_{\Omega} +
\langle g_{p}, \bfwh \cdot \bfn\rangle_{\partial_{p} \Omega} + 
(\psi, \ph)_\Omega.
\end{aligned}
\end{equation*}
First, we focus our attention on the left-hand side term $\mathcal{N}$. Using the computability of the coupling term in~\eqref{eq:coupling_term_discrete} we have 
\begin{equation}\label{eq:N_t:1}
\begin{aligned}
\mathcal{N}= & 
\sum_{E\in\Th} a_h^E(\partialt\bfsigmah,\bfsigmah) +
\sum_{E\in\Th} \left(\alpha \partialt \ph, \frac{\tr(\bs{\Pi}_{\rm s}^E\bfsigmah)}{3\kappa}\right)_E +\\ &
\sum_{E\in\Th} \left(\alpha  \ph, \frac{\tr(\partialt\bs{\Pi}_{\rm s}^E\bfsigmah)}{3\kappa}\right)_E +
\left(
(s_0 + \alpha^2 \kappa^{-1} )
\partialt \ph, \ph
\right)_\Omega.
\end{aligned}
\end{equation}
By integrating by parts (with respect to time) the first, second, and fourth terms of~\eqref{eq:N_t:1}, and recalling the definition of $a_h^E(\cdot,\cdot)$ and~\eqref{eq:compliance}, we get
\begin{equation*}
    \begin{aligned}
2\mathcal{N} =&    
 \ddt
\sum_{E\in\Th}\left(
\left(
\frac{\bdev(\bs{\Pi}_{\rm s}^E\bfsigmah)}{2\mu}, \bs{\Pi}_{\rm s}^E\bfsigmah
\right)_E 
 +\left(
\frac{\tr(\bs{\Pi}_{\rm s}^E\bfsigmah)}{9\kappa}
, \tr(\bs{\Pi}_{\rm s}^E\bfsigmah)\right)_E
\right)+
 \\ &
 \ddt
\sum_{E\in\Th}
s_1^E((\bfI -\bs{\Pi}_{\rm s}^E) \bfsigmah,(\bfI -\bs{\Pi}_{\rm s}^E) \bfsigmah) +
\\& \ddt
\sum_{E\in\Th}
2\left(
\alpha \ph, \frac{\tr(\bs{\Pi}_{\rm s}^E\bfsigmah)}{3\kappa}
\right)_E
+ %\\&
\ddt\left(
\left(
\alpha^2 \kappa^{-1} 
\ph, \ph
\right)_\Omega
+
\left(
s_0 
\ph, \ph
\right)_\Omega
\right).
    \end{aligned}
\end{equation*}
Hence, rearranging the previous terms we obtain
\begin{equation*}
    \begin{aligned}
        2\mathcal{N} =& 
   \ddt\sum_{E\in\Th}\left(\norm[E]{(2\mu)^{-\onehalf} \bdev(\bs{\Pi}_{\rm s}^E\bfsigmah)}^2
     +
     s_1^E((\bfI -\bs{\Pi}_{\rm s}^E) \bfsigmah,(\bfI -\bs{\Pi}_{\rm s}^E) \bfsigmah)\right)+\\& \ \ddt
     \sum_{E\in\Th}\norm[E]{\kappa^{-\onehalf}    \left(3^{-1}\tr(\bs{\Pi}_{\rm s}^E\bfsigmah)+ \alpha \ph \right)}^2 + \ddt \norm[\Omega]{s_0^{\onehalf} \ph}^2.
    \end{aligned}
\end{equation*}
After exploiting the stability of the bilinear forms $c_h(\cdot,\cdot)$ in~\eqref{eq:coercivity_continuity_bilinearforms} and  $s_1^E(\cdot,\cdot)$ in~\eqref{eq:stability_estimates_stab}, we integrate in time from $t_0$ to $t\le t_f$ to get
\begin{equation}\label{eq:main_estimate_stability}
    \begin{aligned}
&\sum_{E\in\Th} \left(
\norm[E]{(2\mu)^{-\onehalf} \bdev(\bs{\Pi}_{\rm s}^E\bfsigmah(t))}^2+
 \gamma_{1,*}\norm[E]{(2\mu)^{-\onehalf}(\bfI -\bs{\Pi}_{\rm s}^E) \bfsigmah(t)}^2
\right)
+\\
&
\sum_{E\in\Th}\norm[E]{\kappa^{-\onehalf}    \left(3^{-1}\tr(\bs{\Pi}_{\rm s}^E\bfsigmah(t))+ \alpha \ph(t) \right)}^2
+\norm[\Omega]{s_0^{\onehalf} \ph(t)}^2 
+\\
&
\ 2C_{v,*}\int_{t_0}^t\norm[\Omega]{\bfK^{-\onehalf}\bfwh(s)}^2\ds
\leq 2\int_{t_0}^t \mathcal{R}(s)\ds
+ \mathcal{N}_{t_0},      
    \end{aligned}
\end{equation}
where $\mathcal{N}_{t_0}$ is the evaluation of the first four terms in the left-hand side of \eqref{eq:main_estimate_stability} at time $t=t_0$ and only depends on the initial data.
\paragraph{Step 2: $L^2$-bounds for the primary variables} We start by establishing  bounds of the $L^2$-norm of the displacement $\bfuh$ and pressure $p_h$. Using the discrete inf-sup condition \cite[Proposition 4.5]{Dassi-Lovadina-Visinoni:2020}, equation~\eqref{problem:semidiscrete_formulation:stress}, along with the computability of the coupling term, equation~\eqref{eq:compliance}, the stability of the discrete bilinear form  $a_h(\cdot,\cdot)$, and the Cauchy-Schwarz inequality, we obtain 
\begin{equation*}
\begin{aligned}
\beta_{el} &\norm[\Omega]{\bfuh}  \leq \sup_{\bftauh\in\bs{\Sigma}_h} \frac{\left(\bfuh,\bdiv \bftauh \right)_\Omega}{\norm[\bs{\Sigma}]{\bftauh}} 
\\ = &
\sup_{\bftauh\in\bs{\Sigma}_h} \frac{\langle \bfg_u, \bftauh\bfn\rangle_{\partial\Omega} -a_h(\bfsigmah,\bftauh)- \sum_{E\in\Th} \left(\alpha\ph/3\kappa,\tr(\bs{\Pi}_{\rm s}^E\bftauh)\right)_E}{\norm[\bs{\Sigma}]{\bftauh}} 
\\ \leq & 
\sup_{\bftauh\in\bs{\Sigma}_h} \frac{|\langle \bfg_u, \bftauh\bfn\rangle_{\partial\Omega}| 
+\sum_{E\in\Th}\left|
\left((2\mu)^{-1}\bdev(\bs{\Pi}_{\rm s}^E\bfsigmah),\bs{\Pi}_{\rm s}^E\bftauh \right)_E \right|}{\norm[\bs{\Sigma}]{\bftauh}}+
\\
&\sup_{\bftauh\in\bs{\Sigma}_h} \frac{\sum_{E\in\Th} \left|\left((3\kappa)^{-1} (\alpha \ph + 3^{-1}\tr(\bs{\Pi}_{\rm s}^E\bfsigmah)) ,\tr(\bs{\Pi}_{\rm s}^E\bftauh)\right)_E +s_1^E(\bfsigmah,\bftauh)\right|
}{\norm[\bs{\Sigma}]{\bftauh}}  
\\ \leq & 
\sup_{\bftauh\in\bs{\Sigma}_h} \frac{
\norm[\onehalf,\partial\Omega]{\bfg_u}\norm[-\onehalf,\partial\Omega]{\bftauh \bfn} +
\sum_{E\in\Th}
\norm[E]{(2\mu)^{-1}\bdev(\bs{\Pi}_{\rm s}^E\bfsigmah)}\norm[E]{\bs{\Pi}_{\rm s}^E\bftauh}
}{\norm[\bs{\Sigma}]{\bftauh}}+
\\ 
&\sup_{\bftauh\in\bs{\Sigma}_h} 
\frac{ \sum_{E\in\Th}
\norm[E]{\kappa^{-1} (\alpha \ph + 3^{-1}\tr(\bs{\Pi}_{\rm s}^E\bfsigmah)}\norm[E]{3^{-1}\tr(\bs{\Pi}_{\rm s}^E\bftauh))}
%\gamma_1^*\norm[E]{(\bfI-\Pi_s^E)\bfsigmah}\norm[E]{(\bfI-\Pi_s^E)\bftauh}
%\norm[E]{(d^2\kappa)^{-1}\tr(\Pi_s^E\bfsigmah)}
%\norm[E]{\tr(\Pi_s^E\bftauh)}
}{\norm[\bs{\Sigma}]{\bftauh}}+\\ &
\sup_{\bftauh\in\bs{\Sigma}_h} 
\frac{
\sum_{E\in\Th}
2\gamma_1^*\norm[E]{(2\mu)^{-1}(\bfI-\bs{\Pi}_{\rm s}^E)\bfsigmah}\norm[E]{(\bfI-\bs{\Pi}_{\rm s}^E)\bftauh}
}{\norm[\bs{\Sigma}]{\bftauh}}.
\end{aligned}
\end{equation*}
Owing to the $H(\rmdiv)$ trace inequality, we infer from the previous bound that
\begin{equation*}
\begin{aligned}
\beta_{el}&\norm[\Omega]{\bfuh} \leq 
 C_{tr} \norm[\onehalf,\partial_u \Omega]{\bfg_u}
+
\sum_{E\in\Th}
\norm[E]{(2\mu)^{-1}\bdev(\bs{\Pi}_{\rm s}^E\bfsigmah)}
+\\
&
\sum_{E\in\Th}
\norm[E]{\kappa^{-1}(\alpha\ph + 3^{-1}\tr(\bs{\Pi}_{\rm s}^E\bfsigmah))}
 +
\sum_{E\in\Th}
2\gamma_1^*\norm[E]{(2\mu)^{-1}(\bfI-\bs{\Pi}_{\rm s}^E)\bfsigmah}.
    \end{aligned}
\end{equation*}
where $C_{tr}>0$ only depends on $\Omega$. Squaring the previous inequality and recalling the assumptions on the model coefficients we have 
\begin{equation}\label{eq:estimate_uh}
\begin{aligned}
&\norm[\Omega]{\bfuh}^2 \leq C_1
\left(\norm[\onehalf,\partial_u \Omega]{\bfg_u}^2
+
\sum_{E\in\Th}
\norm[E]{\kappa^{-\onehalf}(\alpha\ph + \tr(\bs{\Pi}_{\rm s}^E\bfsigmah)/3)}^2
\right.
+\\ 
&\qquad\left.
\sum_{E\in\Th}
\norm[E]{(2\mu)^{-\onehalf}\bdev(\bs{\Pi}_{\rm s}^E\bfsigmah)}^2
 + \sum_{E\in\Th}
\gamma_{1,*}\norm[E]{(2\mu)^{-\onehalf}(\bfI-\bs{\Pi}_{\rm s}^E)\bfsigmah}^2\right),
\end{aligned}
\end{equation}
where $C_{1}=4\beta_{el}^2\max\left(C_{tr}^2,\ %d(2\underline{\mu})^{-1}, 
2(\gamma_1^*)^2 (\underline{\mu}\gamma_{1,*})^{-1}\right)$.
By adopting a similar argument, we can derive the following result for the $L^2$-norm of $\ph$:
\begin{equation*}
\begin{aligned}
\beta_{f}\norm[\Omega]{\ph} & \leq \sup_{\bfzh\in\bfW_h} \frac{\left(\ph,\div \bfzh \right)_\Omega}{\norm[\bf W]{\bfzh}} \\
&= 
\sup_{\bfzh\in\bfW_h} \frac{\langle g_p, \bfzh \cdot \bfn \rangle_{\partial\Omega} - c_h(\bfwh,\bfzh)}{\norm[\bf W]{\bfzh}}
\\&\leq 
\sup_{\bfzh\in\bfW_h} \frac{\norm[\onehalf,\partial_p\Omega]{g_p} \norm[-\onehalf,\partial\Omega]{\bfzh \cdot \bfn}+ (1+\gamma_{2}^*)\norm[\Omega]{\bs{K}^{-\onehalf}\bfwh} \norm[\Omega]{\bs{K}^{-\onehalf}\bfzh}}{\norm[\bf W]{\bfzh}}
\\ &
\leq C_{tr}
\norm[\onehalf,\partial_p\Omega]{g_p} + (1+\gamma_{2}^*)\underline{K}^{-\onehalf}\norm[\Omega]{\bs{K}^{-\onehalf}\bfwh}.
\end{aligned}
\end{equation*}
By squaring both sides of the previous inequality, integrating in time from $t_0$ to $t$, and introducing $C_{2}=\beta_f^2\max\left(2 C_{tr}^2,\ \frac{(1+\gamma_{2}^*)^2}{\underline{K}C_{v,*}}\right)$, we obtain
%\begin{equation*}
%\norm[\Omega]{\ph}^2 \leq 2\beta_f \left(C_{tr}^2 \norm[\onehalf,\partial_p\Omega]{g_p}^2 + (C_{v}^*)^2 \underline{K}^{-1}\norm[\Omega]{\bs{K}^{-\onehalf}\bfwh}^2\right),
%\end{equation*}
\begin{equation}\label{eq:estimate_ph}
\int_{t_0}^t \norm[\Omega]{\ph(s)}^2\ds \leq C_{2} \left( \int_{t_0}^t
\norm[\onehalf,\partial_p\Omega]{g_p(s)}^2 + {2C_{v,*}} \norm[\Omega]{\bs{K}^{-\onehalf}\bfwh(s)}^2\ds \right).
\end{equation}
%% -> space

\paragraph{Step 3: $\bfH(\rm{div})$-bounds for the dual variables}
For the estimate of the stress divergence, we simply test equation~\eqref{problem:semidiscrete_formulation:displacement}  with $v_h=\div \bfsigmah / (2\mu)$ and apply the Cauchy--Schwarz inequality to get
$$
\norm[\Omega]{\div\left((2\mu)^{-\onehalf}\bfsigmah\right)}^2\le (2\underline{\mu})^{-1}\norm[\Omega]{\bfb}^2.
$$
Therefore, using the $\mathrm{dev}$-$\mathrm{div}$ Lemma (cf. \cite[Proposition 9.1.1]{boffi2013} and \cite[Lemma A.4]{Visinoni:25}) followed by the previous bound, we infer the existence of a positive constant $C_D$ depending only on the shape regularity such that
\begin{equation}\label{eq:estimate_sigmah}
\begin{aligned}
C_D\norm[\bs{\Sigma}]{(2\mu)^{-\onehalf}\bfsigmah}^2 &\le
 \norm[\Omega]{(2\mu)^{-\onehalf}\bdev(\bfsigmah)}^2 + 
\norm[\Omega]{\div\left((2\mu)^{-\onehalf}\bfsigmah\right)}^2
\\&\le  \sum_{E\in\Th}
\norm[E]{(2\mu)^{-\onehalf}\bdev(\bs{\Pi}_{\rm s}^E\bfsigmah)}^2+
(2\underline{\mu})^{-1}\norm[\Omega]{\bfb}^2 
\\
&\quad+C_{s,*}^{-1}\sum_{E\in\Th}\gamma_{1,*}\norm[E]{(2\mu)^{-\onehalf}(\bfI -\bs{\Pi}_{\rm s}^E) \bfsigmah}^2.
\end{aligned}
\end{equation}
We now proceed by deriving an upper bound for the filtration displacement $\overline{\bs{w}}_h$ defined as the time integral of the Darcy velocity, i.e. $\overline{\bs{w}}_h(t)=\int_{t_0}^t \bfwh(s) \rm{d}s$.  
First, it follows from the Cauchy--Schwarz inequality that  
\begin{equation}\label{eq:estimate1_whL2}
\begin{aligned}
\hspace{-2mm}
\norm[\Omega]{\overline{\bs{w}}_h(t)}^2  
&= \int_\Omega \left(\int_{t_0}^t\bfwh(s)\ \rm{d}s\right)^2 
\le (t-t_0) \int_\Omega \int_{t_0}^t \bfwh(s)^2\ {\rm{d}}s
\\&= (t-t_0) \int_{t_0}^t \norm[\Omega]{\bfwh(s)}^2 {\rm{d}s}
\le (t-t_0)\overline{K} 
\int_{t_0}^t \norm[\Omega]{\bs{K}^{-\onehalf}\bfwh(s)}^2{\rm{d}}s.
\end{aligned}
\end{equation}
Then, we integrate in time equation~\eqref{problem:semidiscrete_formulation:pression} from $t_0$ to $t$ and choose $\qh = \div\overline{\bs{w}}_h$ as a test function. Then, using the Cauchy-Schwarz  inequality, we have that 
\begin{equation*}
    \begin{aligned}
        &\norm[\Omega]{\div\overline{\bs{w}}_h(t)} \leq 
        \norm[\Omega]{s_0\ph(t)}
        + \left(\sum_{E\in\Th}\norm[\Omega]{\frac{\alpha}{\kappa} \left(\tr(\Pi_s^E\bfsigmah(t))/3 + \alpha\ph(t)\right)}^2\right)^\onehalf + \\
        &\qquad \norm[\Omega]{s_0\ph(t_0)}
        + \left(\sum_{E\in\Th}\norm[\Omega]{\frac{\alpha}{\kappa} \left(\tr(\Pi_s^E\bfsigmah(t_0))/3 + \alpha\ph(t_0)\right)}^2\right)^\onehalf +\norm[\Omega]{\overline{\psi}(t)},
    \end{aligned}
\end{equation*}
where $\overline{\psi}(t)=\int_{t_0}^t \psi(s)\ \rm{d}s$.
Owing to the definitions of the storativity coefficient $s_0$ and bulk modulus $\kappa$, we can assume without loss of generality that $s_0\le C_3 \kappa^{-1}\le C_3 \underline{\mu}^{-1}$, with $C_3$ independent of the model and discretization parameters. Thus, squaring the previous estimate, it is inferred that 
\begin{equation*}
        \norm[\Omega]{\div\overline{\bs{w}}_h}^2 \leq \frac{4 C_3}{\underline{\mu}}\left(
        \norm[\Omega]{s_0^{\onehalf}\ph}^2
        + \sum_{E\in\Th}\norm[\Omega]{\frac{\tr(\Pi_s^E\bfsigmah) + d\alpha\ph}{3\kappa^{\onehalf}}}^2 + \mathcal{N}_{t_0}\right)
        +4\norm[\Omega]{\overline{\psi}}^2,
\end{equation*}
with $\mathcal{N}_{t_0}$ defined in~\eqref{eq:main_estimate_stability}. 
Summing the previous estimate and \eqref{eq:estimate1_whL2}, we have
\begin{equation}~\label{eq:estimate_wh}
\begin{aligned}
\norm[\bfW]{\overline{\bs{w}}_h}^2
&\leq C_{4} \left(\norm[\Omega]{s_0^{\onehalf}\ph}^2
        + \sum_{E\in\Th}
\norm[E]{\kappa^{-\onehalf}(\alpha\ph + \tr(\bs{\Pi}_{\rm s}^E\bfsigmah)/3)}^2 + \mathcal{N}_{t_0}  \right. \\
& \left. \qquad\;
+(t-t_0)\int_{t_0}^t \norm[\Omega]{\psi(s)}^2\ds 
+ 2C_{v,*} \int_{t_0}^t\norm[\Omega]{\bs{K}^{-\onehalf}\bfwh(s)}^2\ds\right),
\end{aligned}
\end{equation}
with $C_{4} = 4\max\left(C_3 \underline{\mu}^{-1}, 1,\ \frac{(t-t_0)\overline{K}}{8 C_{v,*}}\right)$.
\paragraph{Step 4: estimate of the right-hand side in~\eqref{eq:main_estimate_stability}}
We aim to bound
\begin{equation}\label{eq:rhs_main_estimate}
\begin{aligned}
    \int_{t_0}^t \mathcal{R}(s)\ds = &
    \underbrace{\int_{t_0}^t
\left(
\bfb(s),\partialt \bfuh(s)
\right)_{\Omega}\ds
}_{=:T_1}+ 
\underbrace{\int_{t_0}^t
(\psi(s), \ph(s))_\Omega\ds
}_{=:T_2}+\\
&
\underbrace{
\int_{t_0}^t \langle 
\partialt\bfg_u(s), \bfsigmah(s)\bfn
\rangle_{\partial_u \Omega}\ds
}_{=:T_3}
+
\underbrace{\int_{t_0}^t
\langle g_{p}(s), \bfwh(s) \cdot \bfn\rangle_{\partial_{p} \Omega}\ds
}_{=:T_4}. 
\end{aligned}
\end{equation}
Integrating by parts (with respect to time), using the Cauchy-Schwarz and Young inequalities with 
$\varepsilon_1 > 0$, and reasoning as in \eqref{eq:estimate1_whL2}, we obtain 
\begin{equation}~\label{eq:t1_rhs}
\begin{aligned}
%\int_{t_0}^t (\bfb(s),\partialt \bfuh)_\Omega\ds
T_1 &= (\bfb(t), \bfuh(t))_\Omega - (\bfb(t_0), \bfuh(t_0))_\Omega + \int_{t_0}^t (\partialt \bfb, \bfuh)_\Omega\ds \\
    &\leq  \norm[\Omega]{\bfb(t)} \norm[\Omega]{\bfuh(t)} + \norm[\Omega]{\bfb(t_0)} \norm[\Omega]{\bfuh(t_0)} 
    + \int_{t_0}^t \frac{\norm[\Omega]{\partialt \bfb(s)}}{(t-t_0)^{-\onehalf}}\ 
    \frac{\norm[\Omega]{\bfuh(s)}}{(t-t_0)^\onehalf}\ds\\
    &\leq  %\varepsilon_1\norm[\Omega]{\bfb(t)}^2+
    \sup_{s\in[t_0,t]}\frac{\norm[\Omega]{\bfuh(s)}^2}{4\varepsilon_1}
    +4\varepsilon_1\norm[\Omega]{\bfb(t_0)}^2
    %+\frac{\norm[\Omega]{\bfuh(t_0)}^2}{2\varepsilon} 
    +8\varepsilon_1(t-t_0) \int_{t_0}^t \norm[\Omega]{\partialt\bfb(s)}^2\ds.
\end{aligned}
\end{equation}
For the second term, we use again the Cauchy-Schwarz and Young inequalities with $\varepsilon_2 > 0$ in order to obtain
\begin{equation}~\label{eq:t2_rhs}
    T_2\leq  \int_{t_0}^t\norm[\Omega]{\psi(s)} \norm[\Omega]{\ph(s)}\ds 
    \leq \frac{\varepsilon_2}2\int_{t_0}^t\norm[\Omega]{\psi(s)}^2\ds 
    + \int_{t_0}^t \frac{\norm[\Omega]{\ph(s)}^2}{2\varepsilon_2}\ds 
\end{equation}
For the terms $T_3$ and $T_4$ related to the non-homogeneous boundary conditions, we use the Cauchy-Schwarz inequality, the $\bfH(\rm{div})$-trace inequality, and the Young inequality with $\varepsilon_3>0$ in order to get
\begin{equation}~\label{eq:t3_rhs}
    \begin{aligned}
        T_3 &\leq \int_{t_0}^t\norm[\onehalf,\partial_u\Omega]{\partialt\bfg_u(s)}  \norm[-\onehalf, \partial_u\Omega]{\bfsigmah(s)\bfn}\ds\\
        &\leq C_{tr} \int_{t_0}^t (t-t_0)^\onehalf \norm[\onehalf,\partial_u\Omega]{\partialt\bfg_u(s)}  \frac{\norm[\bs{\Sigma}]{\bfsigmah(s)}}{(t-t_0)^\onehalf}\ds\\
        &\leq \frac{\varepsilon_3}4 \sup_{s\in[t_0,t]}\norm[\bs{\Sigma}]{\bfsigmah(s)}^2+ \frac{C_{tr}^2(t-t_0)}{\varepsilon_3} \int_{t_0}^t\norm[\onehalf,\partial_u\Omega]{\partialt\bfg_u(s)}^2 \ds.
    \end{aligned}
\end{equation}
Finally, for the term $T_4$, we integrate by parts in time and proceed as for term $T_1$ to infer
\begin{equation}~\label{eq:t4_rhs}
    \begin{aligned}
        T_4 &
        =\langle g_p(t), \overline{\bs{w}}_h(t)\cdot \bfn \rangle -\langle g_p(t_0), \overline{\bs{w}}_h(t_0)\cdot \bfn \rangle +
        \int_{t_0}^t \langle\partialt g_p(s), \overline{\bs{w}}_h(s)\cdot \bfn \rangle\ \ds\\
        %&\leq 
        %\norm[\onehalf, \partial\Omega]{g_p(t)}
        %\norm[-\onehalf, \partial\Omega]%{\overline{\bs{w}}_h(t)\cdot \bfn}
        %\hspace{-1mm}+\hspace{-1mm}
        %\int_{t_0}^t \hspace{-1mm}\norm[\onehalf, \partial\Omega]{\partialt g_p(s)}
        %\hspace{-0.5mm}\norm[-\onehalf, \partial\Omega]{\overline{\bs{w}}_h(s)\cdot \bfn}\hspace{-0.5mm}\ds 
        %\\
        &\leq \frac{\varepsilon_4}2 \sup_{s\in[t_0,t]}\norm[\bs{W}]{\overline{\bs{w}}_h(s)}^2\hspace{-0.5mm} +
        \frac{2 \norm[\onehalf, \partial\Omega]{g_p(t_0)}^2}{C_{tr}^{-2}\varepsilon_4}
        + \frac{2 (t-t_0)}{C_{tr}^{-2}\varepsilon_4} \hspace{-0.5mm} \int_{t_0}^t\hspace{-1mm}\norm[\onehalf,\partial_p\Omega]{\partialt g_p(s)}^2 \hspace{-0.5mm}\ds.
    \end{aligned}
\end{equation}
We can also derive an estimate for the term $\mathcal{N}_{t_0}$ appearing in \eqref{eq:main_estimate_stability} by testing the initial problem \eqref{eq:discrete_initial} with $(\bftauh, \bfvh, q_h)=(\bfsigmah(t_0), \bfuh(t_0), p_h(t_0))$. Then, reasoning as in the previous steps, it is inferred that 
\begin{equation*}
\mathcal{N}_{t_0} \hspace{-0.5mm}\le \hspace{-0.5mm}
\frac{\norm[\Omega]{\bfuh(t_0)}^2}{2\varepsilon_1}
    \hspace{-0.5mm}+\hspace{-0.5mm}2\varepsilon_1 \hspace{-0.5mm}\norm[\Omega]{\bfb(t_0)}^2\hspace{-0.5mm}
    +\hspace{-0.5mm}\frac{\varepsilon_3}2\hspace{-0.5mm}\norm[\bs{\Sigma}]{\bfsigmah(t_0)}^2 \hspace{-0.5mm}+ \hspace{-0.5mm}\frac{2C_{tr}^2}{\varepsilon_3} \norm[\onehalf,\partial_u\Omega]{\bfg_u(t_0)}^2
    \hspace{-0.5mm}+ \hspace{-1mm}\norm[\Omega]{\eta_0}^2\hspace{-0.5mm}.
\end{equation*}
Now, we sum the terms $T_i$, for $i=1,\dots,4$, namely inequalities~\eqref{eq:t1_rhs},~\eqref{eq:t2_rhs},~\eqref{eq:t3_rhs}, and~\eqref{eq:t4_rhs} for $t=t_f$ and combine with the previous bound on $\mathcal{N}_{t_0}$ to get
\begin{equation}\label{eq:tfinale_rhs}
\begin{aligned}
\int_{t_0}^{t_f} 2\mathcal{R}(s)\ds + \mathcal{N}_{t_0} &\leq 
\sup_{s\in[t_0,t_f]}\frac{\norm[\Omega]{\bfuh(s)}^2}{\varepsilon_1}+
\int_{t_0}^{t_f} \frac{\norm[\Omega]{\ph(s)}^2}{\varepsilon_2}\ds+
\varepsilon_3\sup_{s\in[t_0,t_f]}\norm[\bs{\Sigma}]{\bfsigmah(s)}^2\\
&\quad
+\varepsilon_4\sup_{s\in[t_0,t_f]}\norm[\bs{W}]{\overline{\bs{w}}_h(s)}^2
+\widetilde{\mathcal{C}}_{\varepsilon}(\bfb, \psi, \bfg_u, g_p, \eta_0),
\end{aligned}
\end{equation}
where we have collected all the terms only depending on the problem data in the quantity $\widetilde{\mathcal{C}}_{\varepsilon}(\bfb, \psi, \bfg_u, g_p, \eta_0)$.
\paragraph{Step 5: stability estimate}
The last step of the proof consists in plugging the bounds \eqref{eq:estimate_uh}, \eqref{eq:estimate_ph}, \eqref{eq:estimate_sigmah}, and \eqref{eq:estimate_wh} obtained in the second and third steps into the estimate \eqref{eq:tfinale_rhs} and then combining the result with~\eqref{eq:main_estimate_stability}. To do so, we take the supremum for $t\in[t_0,t_F]$ in both side of~\eqref{eq:main_estimate_stability} and set 
$$
\varepsilon_1 = 8C_1,\quad
\varepsilon_2 = 4C_2,\quad
\varepsilon_3 = \frac{C_D C_{s,*}}{8\overline{\mu}},\quad
\varepsilon_4 = \frac1{4C_4}
$$
in \eqref{eq:tfinale_rhs}. Rearranging the resulting inequality, we infer
\begin{equation}\label{eq:estimate_stability_step5}
    \begin{aligned}
&\sup_{t\in[t_0,t_f]}\Big[\sum_{E\in\Th} \left(
\norm[E]{(2\mu)^{-\onehalf} \bdev(\bs{\Pi}_{\rm s}^E\bfsigmah(t))}^2+
 \gamma_{1,*}\norm[E]{(2\mu)^{-\onehalf}(\bfI -\bs{\Pi}_{\rm s}^E) \bfsigmah(t)}^2+
\right.\\
&
\left.\qquad\qquad\qquad
\norm[E]{\kappa^{-\onehalf} \left(3^{-1}\tr(\bs{\Pi}_{\rm s}^E\bfsigmah(t))+ \alpha \ph(t) \right)}^2\right) 
+\norm[\Omega]{s_0^{\onehalf} \ph(t)}^2 \Big]
+\\
&
\quad \int_{t_0}^{t_f}\norm[\Omega]{\bfK^{-\onehalf}\bfwh(s)}^2\ds
\leq C\left((t_f-t_0)\mathcal{C}(\bfb, \psi, \bfg_u, g_p) +
\mathcal{C}_0(\bfb, \bfg_u, g_p, \eta_0)\right),      
\end{aligned}
\end{equation}
where the constant $C>0$ appearing in the previous estimate is independent of $s_0$, $\lambda$, the time interval $[t_0,t_f]$, and the discretization parameters, and the quantities $\mathcal{C}(\bfb, \psi, \bfg_u, g_p)$ and $\mathcal{C}_0(\bfb, \bfg_u, g_p, \eta_0)$ are defined in \eqref{eq:stab.rhs}.

The conclusion follows using again \eqref{eq:estimate_uh}, \eqref{eq:estimate_ph}, and \eqref{eq:estimate_sigmah} to observe that the left-hand side of \eqref{eq:estimate_stability_step5} yields an upper bound for the terms appearing in the stability estimate \eqref{eq:stability_theorem}.
\end{proof}

For the error analysis, we require additional regularity of the solutions to problem \eqref{problem:variationalFormulation}: in particular, for any $t\in(t_0,t_f]$, we assume $\bfsigma(t)\in\bs{\Sigma}\cap\bbH^{1}(\Omega)$ such that $\bdiv \bfsigma(t)\in\bfH^1(\Omega)$, $\bfu(t)\in\bfU\cap \bfH^1(\Omega)$, $\bfw(t) \in\bfW\cap \bfH^1(\Omega)$ such that $\div \bfw (t) \in H^1(\Omega)$ and $p(t)\in Q\cap H^1(\Omega)$. We start introducing the estimates of the VEM interpolation operators $\bfsigmaI\in\bs{\Sigma}_h$ of $\bfsigma$, $\bfuI\in\bfU_h$ of $\bfu$, $\bfwI\in\bfW_h$ of $\bfw$ and $\pI\in Q_h$ of $p$, see~\cite{Dassi-Lovadina-Visinoni:2020,Visinoni:25,Brezzi-Falk-Marini:2014}.
\begin{lemma}\label{lemma:interpolation_estimates}
Under the mesh assumptions A1, A2 and A3, and the regularity hypotheses stated above, there exists a positive constant $C$ independent of the mesh size $h$ such that, for every element $E\in\Th$, it holds
\begin{equation*}
\begin{aligned}
\norm[E]{\bfsigma-\bfsigmaI} &\leq C h \seminorm[1,E]{\bfsigma} \quad\; \text{and} \quad \norm[E]{\bdiv\bfsigma-\bdiv\bfsigmaI} \leq C h \seminorm[1,E]{\bdiv\bfsigma},\\
\norm[E]{\bfw-\bfwI} &\leq C h  \seminorm[1,E]{\bfw}\quad  \text{and} \quad 
\norm[E]{\div\bfw- \div\bfwI}\leq C h \seminorm[1,E]{\div \bfw},\\
\norm[E]{\bfu-\bfuI} &\leq C h |\bfu|_{1,E},
\quad  \text{and} \quad 
\norm[E]{p-\pI}\leq C h \seminorm[1,E]{p}.
\end{aligned}
\end{equation*}
\end{lemma}
\begin{lemma}\label{lemma:interpolation_estimates2}
Under the mesh assumptions A1, A2 and A3, and the regularity hypotheses stated above, there exists  $C>0$ independent of $\lambda$ and $h$ such that
\begin{equation}
\begin{aligned}
a_h(\bfsigma,\bftauh) - (\mathcal{A}\bfsigma, \bftauh)_{\Omega} &\le C h \norm[\bbH^1(\Omega)]{\bfsigma}\norm[\Omega]{\bftauh} \quad &\forall \bftauh \in\bs\Sigma_h,\\
c_h(\bfw, \bfzh) - (\bfK^{-1} \bfw , \bfzh)_{\Omega} &\le C h\norm[\bfH^1(\Omega)]{\bfw}\norm[\Omega]{\bfzh} \quad &\forall \bfzh \in \bfW_h.
\end{aligned}
\end{equation}
\end{lemma}
\begin{proof}
The proof directly follows from~\cite[Proposition 5.7]{Artioli-DeMiranda-Lovadina-Patruno:2017}, \cite[Proposition 5.2]{Brezzi-Falk-Marini:2014}, the bounds in~\eqref{eq:stability_estimates_stab}, and Lemma~\ref{lemma:interpolation_estimates}.
\end{proof}
We remark that the global VE interpolation operators are assembled from the local contributions; hence, for simplicity, we use the same notation for both operators. For all $t\in[t_0,t_f]$, we additionally split the errors as follows
\begin{equation*}
\begin{aligned}
\errsigma(t) &:= \bfsigma(t) -\bfsigmah(t) = (\bfsigma - \bfsigmaI)(t) +(\bfsigmaI -\bfsigmah)(t)  := \errsigmaI(t) + \errsigmaA(t),\\
\erru(t) &:=\bfu(t) -\bfuh(t) = (\bfu -\bfuI)(t) +(\bfuI -\bfuh)(t)  := \erruI(t) + \erruA(t),\\
\errw(t) &:=\bfw(t) -\bfwh(t) = (\bfw -\bfwI)(t) +(\bfwI -\bfwh)(t)  := \errwI(t) + \errwA(t),\\
\errp(t) &:=p(t) -\ph(t)= (p - \pI)(t) +(\pI -\ph)(t)  := \errpI(t) + \errpA(t)
\end{aligned}
\end{equation*}
and derive the error equations for 
$(\errsigmaA , \erruA, \errwA, \errpA)\in\bs{\Sigma}_h\times\bfU_h\times\bfW_h\times Q_h$. 
First, from the definition of the VE interpolants, we infer that
\begin{equation*}
    \begin{aligned}
a_h(\errsigmaI, \bftauh)+ (\erruI, \bdiv \bftauh)_\Omega  &= 0 \,\, 
&&\forall \bftauh\in\bs{\Sigma}_h,\\
(\bdiv \errsigmaI, \bfvh)_\Omega &= 0\,\, &&\forall \bfvh\in\bfU_h,\\
c_h(\errwI, \bfzh)  - (\errpI, \div \bfzh)_\Omega &= 0\, &&\forall \bfzh\in\bfW_h,\\
(\partialt \errsigmaI, \alpha \mathcal{A}\bfI\qh)_\Omega+ (\div \errwI, \qh)_\Omega + ((s_0+\alpha^2 \kappa^{-1})\partialt \errpI, q_h)_{\Omega} &=  0 \, &&\forall \qh\in Q_h,
\end{aligned}
\end{equation*}
where the last identity follows from \eqref{eq:coupling_term_discrete} and \eqref{eq:projection:elasticity_rhs}. Thus, plugging together the previous system and \eqref{problem:semidiscrete_formulation}, we obtain for all $t\in (t_0,t_f]$
\begin{equation}\label{eq:error_equations}
\begin{aligned}
a_h(\errsigmaA(t), \bftauh)+ (\erruA(t), \bdiv \bftauh)_\Omega + (\alpha \mathcal{A}\bfI\errpA(t), \bftauh)_\Omega  &= R_1^t(\bftauh),\\
(\bdiv \errsigmaA(t), \bfvh)_\Omega &= 0,\\
c_h(\errwA(t), \bfzh)  - (\errpA(t), \div \bfzh)_\Omega &= R_2^t(\bfzh), \\
(\partialt \errsigmaA(t), \alpha \mathcal{A}\bfI\qh)_\Omega+ (\div \errwA(t), \qh)_\Omega + ((s_0+\alpha^2 \kappa^{-1})\partialt \errpA(t), q_h)_{\Omega} &=  0, 
\end{aligned}
\end{equation}
where, using \eqref{problem:variationalFormulation:stress} and \eqref{problem:variationalFormulation:velocity}, it is inferred that
$$
\begin{aligned}
R_1^t(\bftauh) &=  a_h(\bfsigma(t),\bftauh) + (\bfu(t),\bdiv \bftauh)_{\Omega} + (\alpha \mathcal{A}\bfI \pI(t), \bftauh)_{\Omega} -\langle\bfg_u(t), \bftauh\bfn\rangle_{\partial_u \Omega}  \\
&= a_h(\bfsigma(t),\bftauh) - (\mathcal{A}\bfsigma(t),\bftauh)_\Omega -
(\alpha \mathcal{A}\bfI \errpI(t), \bftauh)_{\Omega}
\qquad\qquad\qquad\text{ and } \\
R_2^t(\bfzh) &= c_h(\bfw(t), \bfzh) + (p(t), \div \bfzh)_{\Omega}- \langle g_p(t), \bfzh\cdot \bfn\rangle_{\partial_p \Omega} \\
&= c_h(\bfw(t), \bfzh) - (\bfK^{-1} \bfw(t), \bfzh)_{\Omega}.
\end{aligned}
$$
%%
%% convergence theorem
%%
The following a priori error estimate is obtained by employing the arguments used in Theorem \ref{theorem:stability} to the error equations
\eqref{eq:error_equations}.
\begin{theorem}[Error estimate]
Let the mesh assumptions A1, A2 and A3, as well as the regularity hypotheses of Lemma \ref{lemma:interpolation_estimates} hold. For any $t\in (t_0, t_f]$, let $(\bfsigma(t),\bfu(t),\bfw(t),p(t))$ be the solution of the continuous problem~\eqref{problem:variationalFormulation} and $(\bfsigmah(t),\bfuh(t),\bfwh(t),\ph(t))\in\bs{\Sigma}_h\times\bfU_h\times\bfW_h\times Q_h$ be the unique solution of the semi-discrete problem~\eqref{problem:semidiscrete_formulation}. Then, the following error bound holds
\begin{equation*}
\begin{aligned}
&\sup_{t\in[t_0,t_f]}\norm[\Omega]{\erru(t)}^2 +
\sup_{t\in[t_0,t_f]}\norm[\boldsymbol{\Sigma}]{\errsigma(t)}^2 +
\int_{t_0}^{t_f}\norm[\Omega]{\errw(s)}^2\ds +
\int_{t_0}^{t_f}\norm[\Omega]{\errp(s)}^2 \ds\\
& \qquad
\leq 
Ch^2 (t_f-t_0) \left(\norm[H^1(\bbH^1(\Omega))]{ \bfsigma}^2  + \norm[H^1(H^1(\Omega))]{p}^2\right) + Ch^2\norm[L^2(\bfH^1(\Omega))]{ \bfw}^2\\ 
&\qquad\qquad +  Ch^2\norm[L^\infty(\bfH^1(\Omega))]{ \bfu}^2 + Ch^2\big(\norm[\bbH^1(\Omega)]{ \bfsigma(t_0)}^2  + \norm[H^1(\Omega)]{p(t_0)}^2\big).   
\end{aligned}
\end{equation*}
where the constant $C>0$ is independent of $h$, $\lambda$ and $s_0$.
\end{theorem}
%%
%% proof
%%
\begin{proof}
We differentiate the first equation of~\eqref{eq:error_equations} with respect to time, and take as test functions $\bftauh=\errsigmaA$, $\bfvh = \partial_t \erruA$, $\bfzh = \errwA$, and $\qh = \errpA$. Reasoning as in the first step of the proof of Theorem~\ref{theorem:stability}, we readily infer
\begin{equation}\label{eq:main_estimate_convergence}
\begin{aligned}
&\sum_{E\in\Th} \left(
\norm[E]{(2\mu)^{-\onehalf} \bdev(\bs{\Pi}_{\rm s}^E\errsigmaA(t))}^2+
 \gamma_{1,*}\norm[E]{(2\mu)^{-\onehalf}(\bfI -\bs{\Pi}_{\rm s}^E) \errsigmaA(t)}^2
\right)
+\\
&
\sum_{E\in\Th}\norm[E]{\kappa^{-\onehalf}    \left(3^{-1}\tr(\bs{\Pi}_{\rm s}^E\errsigmaA(t))+ \alpha \errpA(t) \right)}^2
+\norm[\Omega]{s_0^{\onehalf} \errpA(t)}^2 
+\\
&
\;\; C_{v,*}\int_{t_0}^t\norm[\Omega]{\bfK^{-\onehalf}\errwA(s)}^2\ds
\leq C_0 \hspace{-1mm}\left(\int_{t_0}^t \hspace{-1mm} \partialt R_1^s(\errsigmaA) \hspace{-0.5mm} + \hspace{-0.5mm} R_2^s(\errwA) \ \ds
+ R_1^{t_0}(\errsigmaA)\hspace{-1mm} \right) \hspace{-1mm},     
    \end{aligned}
\end{equation}
with $C_0>0$ independent of the mesh size $h$.
Using Lemma~\ref{lemma:interpolation_estimates} and Lemma~\ref{lemma:interpolation_estimates2} followed by the Cauchy-Schwarz and Young inequalities, we get
\begin{equation*}
   \begin{aligned}
        \int_{t_0}^t \partialt R_1^s(\errsigmaA)&~\ds 
        \leq  \tilde{C}_1 \int_{t_0}^t h \norm[\bbH^1(\Omega)]{\partialt \bfsigma(s)} \norm[\Omega]{\errsigmaA(s)} + \norm[\Omega]{\partialt \errpI(s)} \norm[\Omega]{\errsigmaA(s)}~ds \\
        \leq & C_1 h \int_{t_0}^t \left( \norm[\bbH^1(\Omega)]{\partialt \bfsigma(s)}  + \norm[H^1(\Omega)]{\partialt p(s)}\right) \norm[\Omega]{\errsigmaA(s)}~ds\\
        \leq & \frac{C_1 h^2(t-t_0)}{4 \varepsilon_1}\left( \norm[H^1(\bbH^1(\Omega))]{ \bfsigma}^2  + \norm[H^1(H^1(\Omega))]{p}^2\right) + 
         \varepsilon_1\sup_{s\in[t_0,t]}\norm[\Omega]{\errsigmaA(s)}^2.
   \end{aligned}
\end{equation*}
Similarly, we bound the second term in the right-hand side of \eqref{eq:main_estimate_convergence} as
\begin{equation*}
\begin{aligned}
  \int_{t_0}^t R_2^s(\errwA)~\ds 
        &\leq  C_2 h \int_{t_0}^t \norm[\bfH^1(\Omega)]{ \bfw(s)} \norm[\Omega]{\errwA(s)}~ds \\
        &\leq \frac{C_2 h^2}{4 \varepsilon_2} \norm[L^2(\bfH^1(\Omega))]{ \bfw}^2 + 
         \varepsilon_2\int_{t_0}^t\norm[\Omega]{\errwA(s)}^2~\ds.
\end{aligned}
\end{equation*}
Finally, the last term in the right-hand side of~\eqref{eq:main_estimate_convergence} is bounded as follows:
\begin{equation*}
   \begin{aligned}
         R_1^{t_0}(\errsigmaA)& 
        \leq  \frac{C_3 h^2}{4 \varepsilon_1}\left( \norm[\bbH^1(\Omega)]{ \bfsigma(t_0)}^2  + \norm[H^1(\Omega)]{p(t_0)}^2\right) + 
         \varepsilon_1\norm[\Omega]{\errsigmaA(t_0)}^2.
   \end{aligned}
\end{equation*}
The constants $\tilde{C}_1$, $C_1$, $C_2$ and $C_3$ are positive and independent of mesh size, $\lambda$ and $s_0$. Proceeding as in the second and third steps of the proof of the stability estimate in Theorem \ref{theorem:stability}, selecting the proper values for $\varepsilon_1$ and $\varepsilon_2$, and taking the supremum over $[t_0,t_f]$, we derive 
\begin{equation*}
\begin{aligned}
&\sup_{t\in[t_0,t_f]}\norm[\Omega]{\erruA(t)}^2 +
\sup_{t\in[t_0,t_f]}\norm[\boldsymbol{\Sigma}]{\errsigmaA(t)}^2 +
\int_{t_0}^{t_f}\norm[\Omega]{\errwA(s)}^2\ds +
\int_{t_0}^{t_f}\norm[\Omega]{\errpA(s)}^2 \ds\\
& \qquad
\leq 
Ch^2 (t_f-t_0) \left(\norm[H^1(\bbH^1(\Omega))]{ \bfsigma}^2  + \norm[H^1(H^1(\Omega))]{p}^2\right) + Ch^2\norm[L^2(\bfH^1(\Omega))]{ \bfw}^2 \\
&\qquad\qquad + Ch^2\big(\norm[\bbH^1(\Omega)]{ \bfsigma(t_0)}^2  + \norm[H^1(\Omega)]{p(t_0)}^2\big).   
\end{aligned}
\end{equation*}
The conclusion follows by applying the triangle inequality and Lemma~\ref{lemma:interpolation_estimates}.
\end{proof}

%%%%%%%%%%%%%%%%%%%%%%%%%%%%%
\section{Numerical Tests} 
\label{sec:numerical}
%%%%%%%%%%%%%%%%%%%%%%%%%%%%%
%%%%%%%% Image path %%%%%%%
\graphicspath{{Figures/}}
%%%%%%%%%%%%%%%%%%%%%%%%%%%
In this section, we present the numerical tests conducted to verify the performance of the proposed method, providing numerical evidence to support the theoretical aspects discussed in the previous section. We begin by introducing the time discretization used in the tests, see Subsection~\ref{subsection:time-discretization}, then focus on assessing the convergence of the method, see Subsection~\ref{subsection:convergence-results}, and finally apply our VE discretization to a more realistic scenario: the footing problem, see Subsection~\ref{subsection:footing-step}.
The implementation is based on the C++ library Vem++~\cite{dassi2023vem}.

%%%%%%%%%%%%%%%%%%%%%%%%%%
\subsection{Time discretization} \label{subsection:time-discretization}
%%%%%%%%%%%%%%%%%%%%%%%%%%
Let $N$ be a positive integer, and let us define the time interval $\Delta t = \frac{t_f - t_0}{N}$ such that each time instant can be written as $t_n = t_0 + n \Delta t$, for $n=0,\dots,N$. Since we are considering the lowest order for spatial discretization, we exploit the backward Euler time-stepping  method for the time discretization.

The fully discrete approximation of~\eqref{problem:variationalFormulation} reads: 
given $(\bfsigma_h^n, \bfu_h^n, \bfw_h^n, p_h^n)\in\bs{\Sigma}_h\times\bfU_h\times \bfW_h\times Q_h$ at time $t_n$, find $(\bfsigma_h^{n+1}, \bfu_h^{n+1}, \bfw_h^{n+1}, p_h^{n+1})\in\bs{\Sigma}_h \times \bfU_h\times \bfW_h\times Q_h$ such that 
{
\small
\begin{equation}\label{problem:fully_discrete_formulation}
\begin{aligned}
a_h(\bfsigmah^{n+1}, \bftauh)
+
\left(
\bfu_h^{n+1},\bdiv \bftauh
\right)_\Omega 
+ 
\left(
\alpha\mathcal{A}\bfI \ph^{n+1} , \bftauh
\right)_\Omega 
&= 
\langle 
\bfg_u^{n+1}, \bftauh\bfn
\rangle_{\partial_u \Omega} \\
-\left(
\bdiv \bfsigmah^{n+1}, \bfvh
\right)_\Omega 
&= (\bfb^{n+1},\bfvh)_\Omega\\
c_h(\bfwh^{n+1},\bfzh)
- \left(\ph^{n+1}, \div\bfzh\right)_\Omega 
&= 
\langle g_{p}^{n+1}, \bfzh \cdot \bfn\rangle_{\partial_{p} \Omega}\\
\left(\bfsigmah^{n+1} , \alpha \mathcal{A}\bfI\qh \right)_\Omega + \Delta t(\div\bfwh^{n+1}, \qh)_\Omega +\left((s_0+\alpha^2\kappa^{-1})\ph^{n+1}, \qh\right)_\Omega &=\\
\Delta t(\psi^{n+1}, \qh)_\Omega +
\left(\bfsigmah^{n} , \alpha \mathcal{A}\bfI\qh \right)_\Omega +
\left((s_0+\alpha^2\kappa^{-1})\ph^{n}, \qh\right)_\Omega
\end{aligned}
\end{equation}
}
for all $\bftauh \in \bs{\Sigma}_h$, $\bfvh \in \bfU_h$, $\bfzh \in \bfW_h$, $\qh \in Q_h$. Therefore, the system can be written in matrix form as follows:
\begin{equation*}
    \left[\begin{array}{cccc}
    \bfA & \bfE^T & 0 & \bfA_I\\
    -\bfE & 0 & 0 & 0 \\
    0 & 0 & \bfM_u & -\bfB^T\\
    \bfA_I & 0 & \Delta t \bfB & \bfM_p+ \bfA_{II}
    \end{array}\right]
    \left[\begin{array}{c} 
    \bfsigma_h^{n+1}\\
    \bfuh^{n+1}\\
    \bfwh^{n+1}\\ 
    \ph^{n+1}
    \end{array}\right]
    =
    \left[\begin{array}{c} 
    \bfg_u\\
    \bfb \\
    \bfg_p\\
    \bfh
    \end{array}\right]
\end{equation*}
where $\bfh = (\bfM_p+\bfA_{II}) {p}_h^{n} + \bs{\psi}\Delta t+ \bfA_I \bfsigma_{h}^{n}$.
In this four-fields formulation we can observe the saddle problem structure of the two individual sub-problems, and a coupling which is symmetric and consists of the $\mathbf{A}_I$ block, unlike the three fields formulation where the coupling term is skew symmetric.

%%%%%%%%%%%%%%%%%%%%%%%%%%%%%%
\subsection{Convergence results} \label{subsection:convergence-results}
%%%%%%%%%%%%%%%%%%%%%%%%%%%%%%
We consider the unite cube $\Omega=[0,1]^3$ as the domain of our problem and we use the following four different tessellations, see Figure~\ref{fig:meshes}:
\begin{itemize}
	\item \texttt{Cube}, a uniform hexahedral mesh; 
	\item \texttt{Tetra}, a Delaunay tetrahedral mesh generated by \texttt{tetgen}~\cite{tetgen};
	\item \texttt{CVT}, a Voronoi tessellation optimized via Lloyd algorithm, obtained by the library \texttt{voro++} \cite{Voro++}; 
	\item \texttt{Rand}, a Voronoi tessellation achieved with random control points and without optimization.
\end{itemize}
%% -> space

\begin{figure}[!ht]
	\centering
	\subfloat{\includegraphics[width=0.25\textwidth]{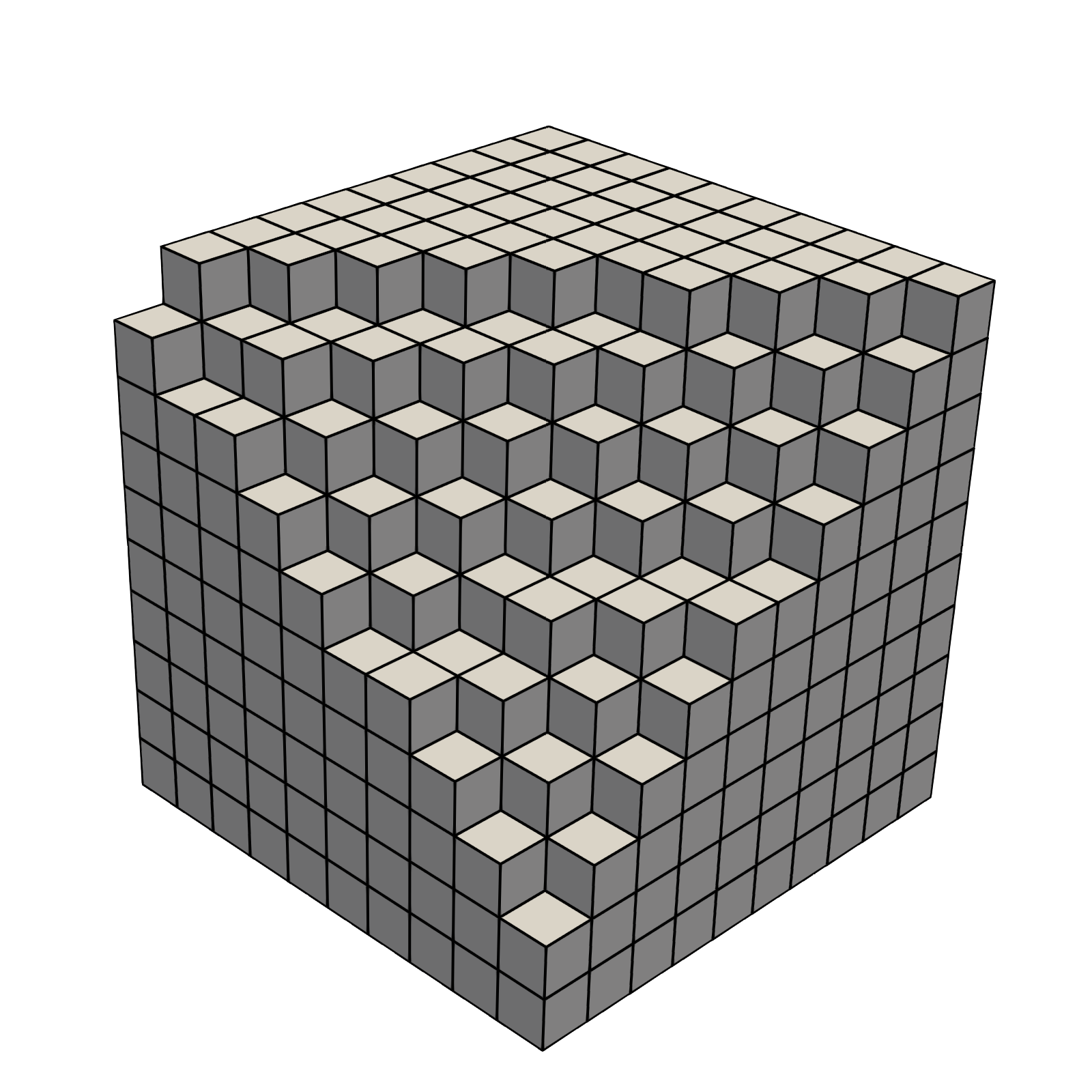}}
	\subfloat{\includegraphics[width=0.25\textwidth]{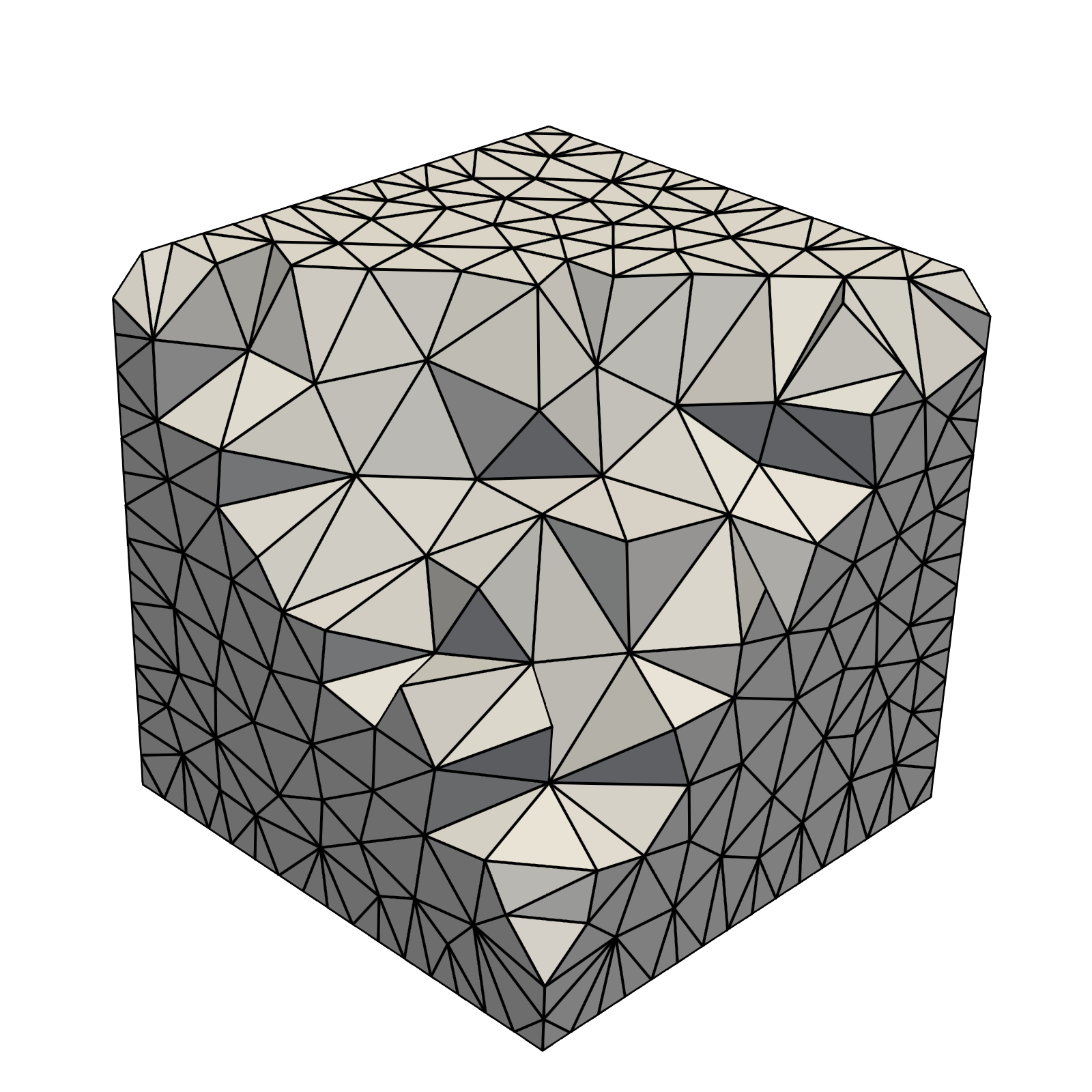}}
	\subfloat{\includegraphics[width=0.25\textwidth]{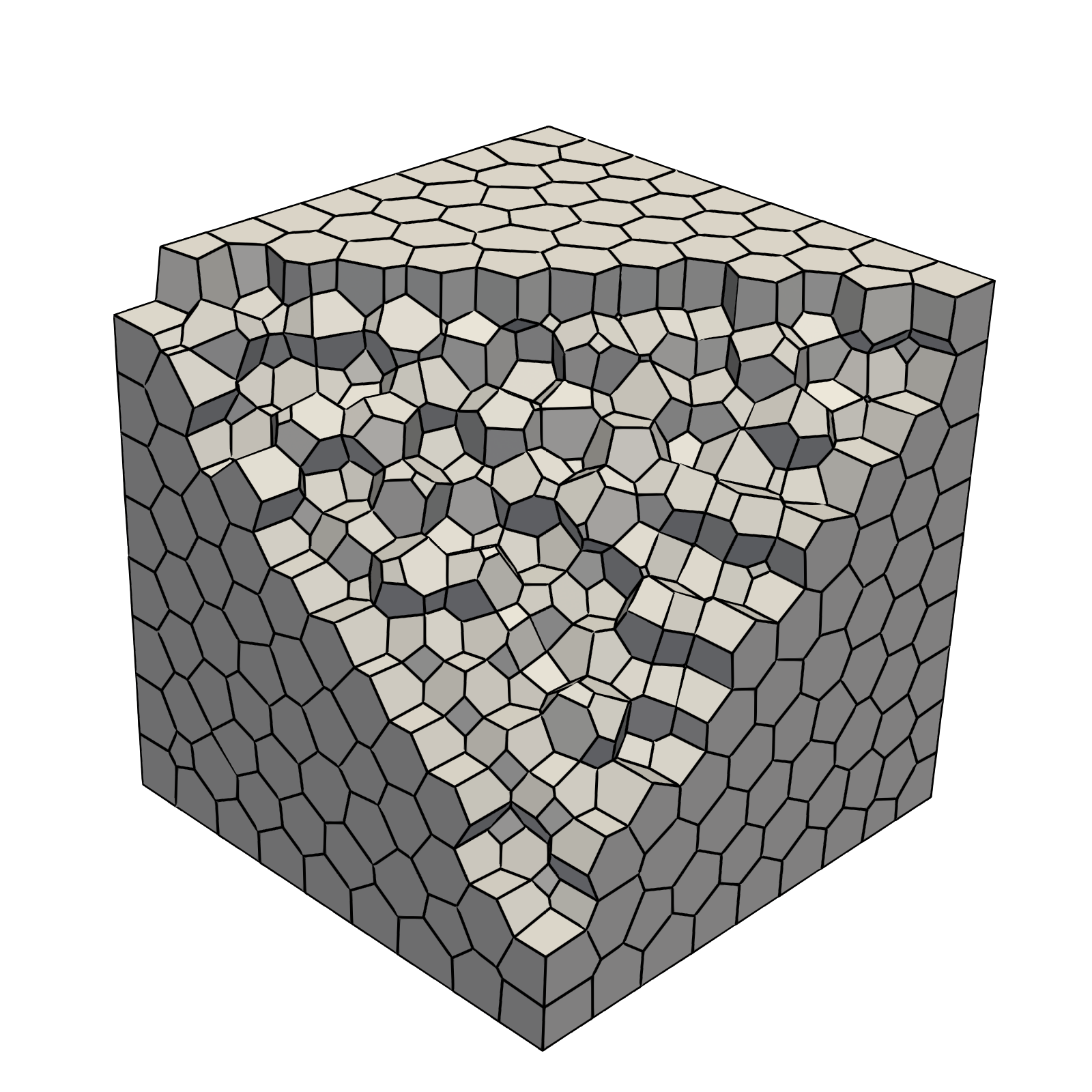}}
	\subfloat{\includegraphics[width=0.25\textwidth]{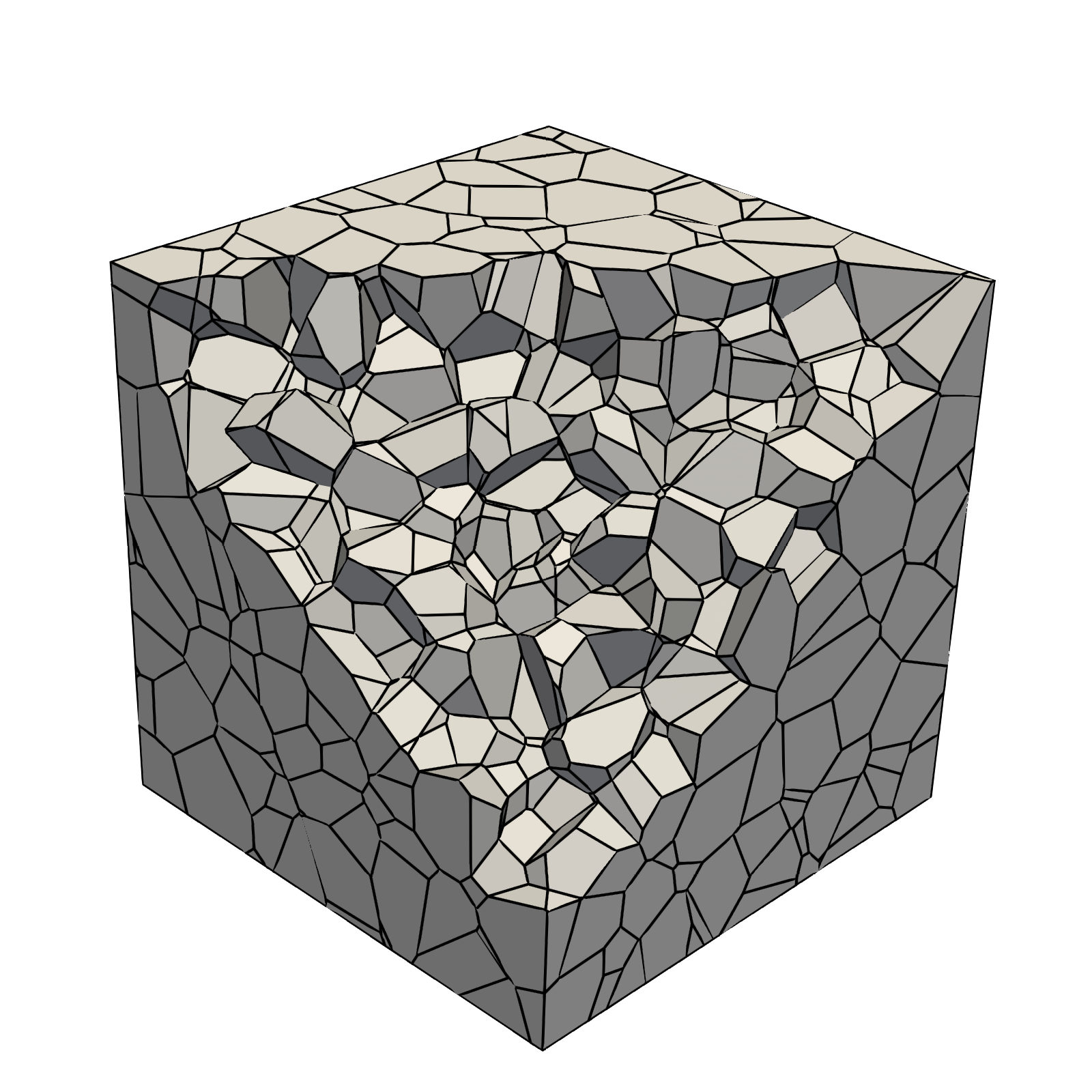}}
	\caption{Overview of adopted meshes for convergence assessment numerical tests.}
	\label{fig:meshes}
\end{figure}
%% -> space

As we can see, the meshes above exhibit two levels of complexity. The first two, \texttt{Cube} and \texttt{Tetra}, are composed of high quality elements and represent a standard choice for a Galerkin method. The two latter instead, \texttt{CVT} and \texttt{Rand}, present an interesting challenge for the robustness of our approach, as they are characterized by elements with some small edges and small faces. For each type of mesh, we estimate the mesh-size as
\begin{equation*}
	h=\frac{1}{N_E}\sum_{E\in\Th} h_E
\end{equation*}
where $N_E$ is the number of the elements in the mesh $\Th$ and $h_E$ is the diameter of the polyhedron $E$.
To assess the spatial accuracy of the method, we take a suitably small time interval $\Delta t$ and compute the following relative errors on a sequence of successively refined meshes:
\begin{equation}\label{eq:L2errors}
\begin{aligned}
&E_{\bfu}:=
\dfrac{\norm[\Omega]{\bfu - \bfuh}}{\norm[\Omega]{\bfu}},	
\qquad 
&E_{\bfsigma, \bs{\Pi}}:=
\dfrac{\left(\sum_{E\in\Th} \norm[E]{\bfsigma -   \bs{\Pi}^{E}_s\bfsigma_h}^2\right)^{\onehalf}}{\norm[\Omega]{\bfsigma}},\\	
&E_{p} :=\dfrac{\norm[\Omega]{p - \ph}}{\norm[\Omega]{p}},
\qquad 
&E_{\bfw, \bs{\Pi}} :=
\dfrac{\left(\sum_{E\in\Th} \norm[E]{\bfw - \bs{\Pi}^{E}\bfwh}^2\right)^{\onehalf}}{\norm[\Omega]{\bfw}},
\end{aligned}    
\end{equation}
where, for simplicity, we decide to consider the discrete solution $(\bfuh, \bfsigmah, \bfwh, \ph)$ at the final time $t_f$. We recall that according to the theory all the above quantities behave as $O(h)$.

%%%%%%%%%%%%%%%%%%%%% -> first test
\paragraph*{Test~1: a compressible material}
%%%%%%%%%%%%%%%%%%%%%
% -> Description 
In this first test, we consider the following manufactured solution of our problem:
\begin{align*}
\bfu = \left( \phi(\bfx,t), \phi(\bfx,t), \phi(\bfx,t) \right)^T, 
\qquad \qquad
p=\phi(\bfx,t),
\end{align*}
where $\phi(\bfx,t)=-xyz(x-1)(y-1)(z-1)(e^t-1)$.
The loading term $\bfb$, the fluid source $\psi$, the boundary conditions, and the initial conditions are computed according to the solution described above. In particular, we consider a problem with pressure and displacement enforced on the whole boundary.
The problem is fully characterized after specifying the following model coefficients
\begin{equation*}
    \alpha=1, \qquad s_0=0.002, \qquad \bfK = \bfI, \qquad  \lambda =1, \qquad \mu =1
\end{equation*}
and time data
\begin{equation*}
    t_{0}=0,\qquad  t_{f}=0.2, \qquad  \Delta t = 0.01.
\end{equation*}
Note that in this case the material is compressible since $\lambda\sim\mu$.

% -> Figures
\begin{figure}[!ht]
\centering
\subfloat{\includegraphics[width=0.4\textwidth]{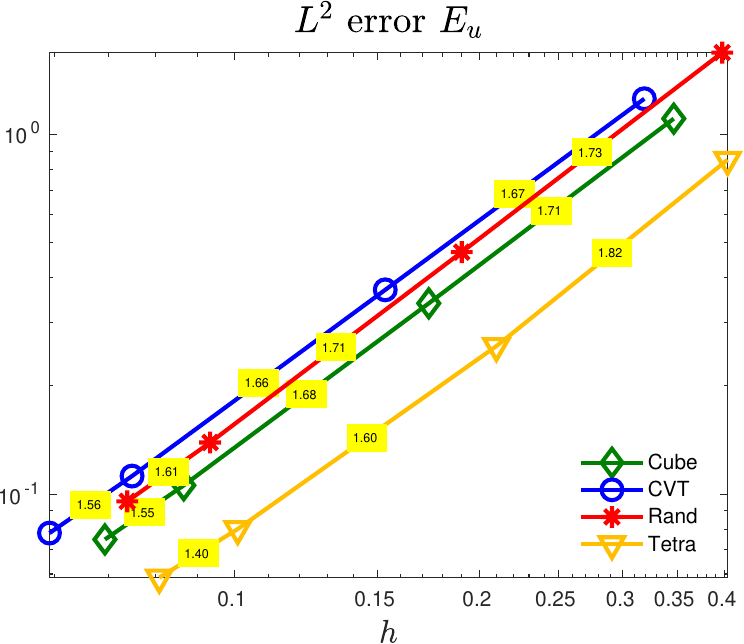}}\quad
\subfloat{\includegraphics[width=0.4\textwidth]{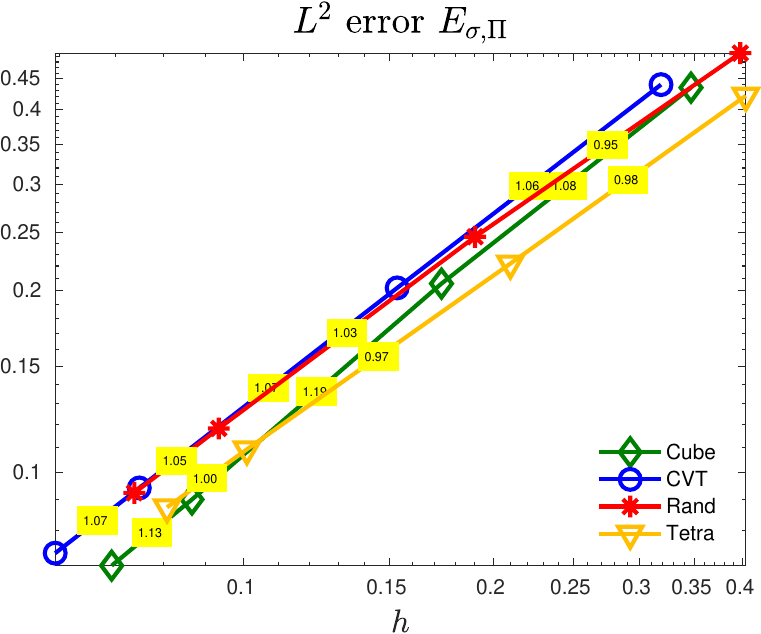}}\\
\subfloat{\includegraphics[width=0.4\textwidth]{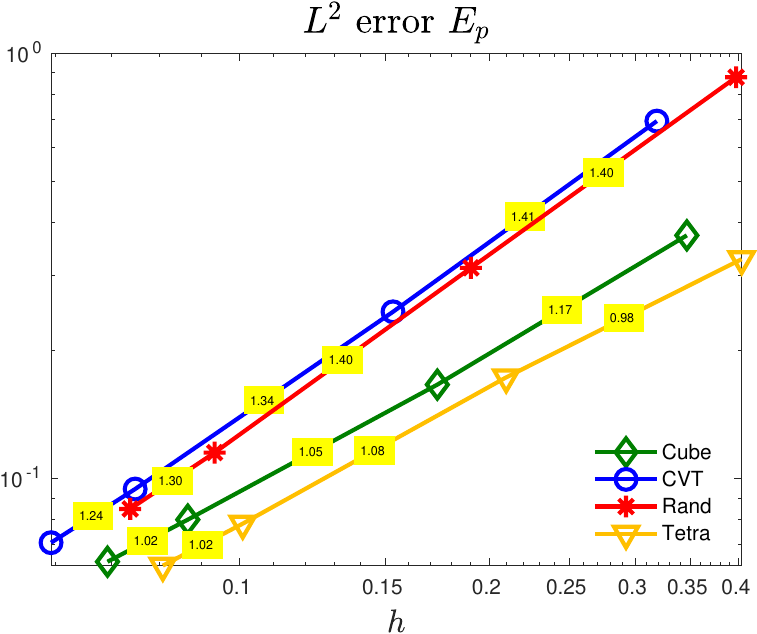}}\quad
\subfloat{\includegraphics[width=0.4\textwidth]{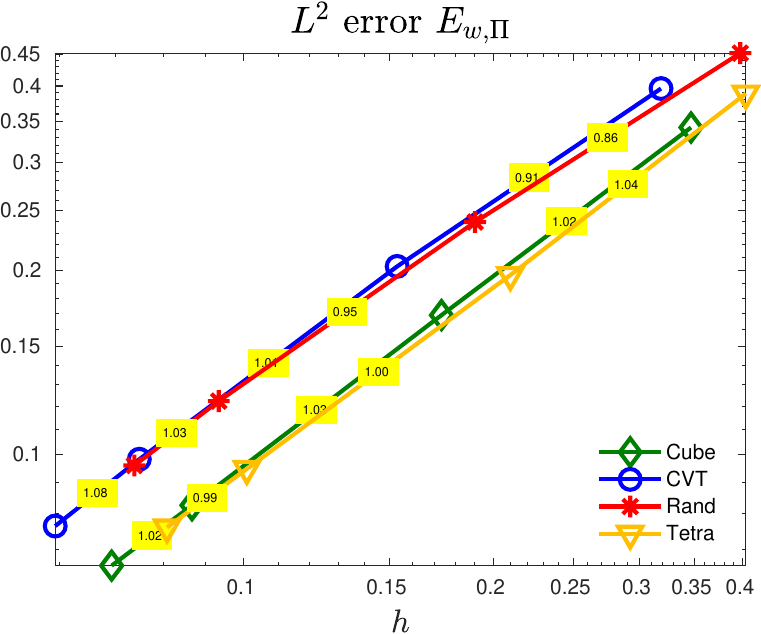}}
\caption{Test 1: a compressible material.}~\label{fig:test1}
\end{figure}

% -> Tests
In Figure~\ref{fig:test1} we report the convergence lines for all the errors. The rates of convergence are computed as 
\begin{equation*}
    \text{rate} = \frac{\log(E_\star/\tilde{E}_\star)}{\log(h/\tilde{h})}
\end{equation*}
where $E_\star$, $\tilde{E}_\star$ denote one of the four errors defined in~\eqref{eq:L2errors} generated on two consecutive meshes of size $h$ and $\tilde{h}$, respectively. As we can notice, for the errors $E_{\bfw,\Pi}$ and $E_{\bfsigma,\Pi}$ the convergence order is approximately 1; for the error $E_{p}$ using \texttt{Cube} and \texttt{Tetra} meshes, the convergence rate is close to 1, while using \texttt{CVT} and \texttt{Rand} meshes the convergence order appears larger than~1; the same happens to $E_{\bfu}$ for all meshes. This fact could be related to the preasymptotic regime, namely the correct order for these unknowns is recovered refining the mesh.

\paragraph*{Test~2: a nearly incompressible material with null storage coefficient}
%%%%%%%
% -> Description 
In this second example, we consider the model parameters and time data as in the previous test except for the following coefficients:
\begin{equation*}
\lambda = 10^6, \quad s_0=0.
\end{equation*}
For the estimate, we consider the following manufactured solution
\begin{equation*}
\begin{aligned}
&\bfu =
    \begin{pmatrix}
    e^{-t}(\cos(2\pi y)\sin(2\pi x)\sin(2\pi z)-\cos(2\pi z)\sin(2\pi x)\sin(2\pi y))+\frac{x^2 e^{-t}}{\lambda+\mu}\\
    e^{-t}(\cos(2\pi z)\sin(2\pi x)\sin(2\pi y)-cos(2\pi x)\sin(2\pi y)\sin(2\pi z))+\frac{y^2 e^{-t}}{\lambda+\mu}\\
    e^{-t}(cos(2\pi x)\sin(2\pi y)\sin(2\pi z)-\cos(2\pi y)\sin(2\pi x)\sin(2\pi z))+\frac{z^2 e^{-t}}{\lambda+\mu}
    \end{pmatrix}\\
&p=e^{-t}\sin(\pi x)\sin(\pi y)\sin(\pi z);
\end{aligned}
\end{equation*}

% -> Figures
\begin{figure}[!ht]
\centering
\subfloat{\includegraphics[width=0.4\textwidth]{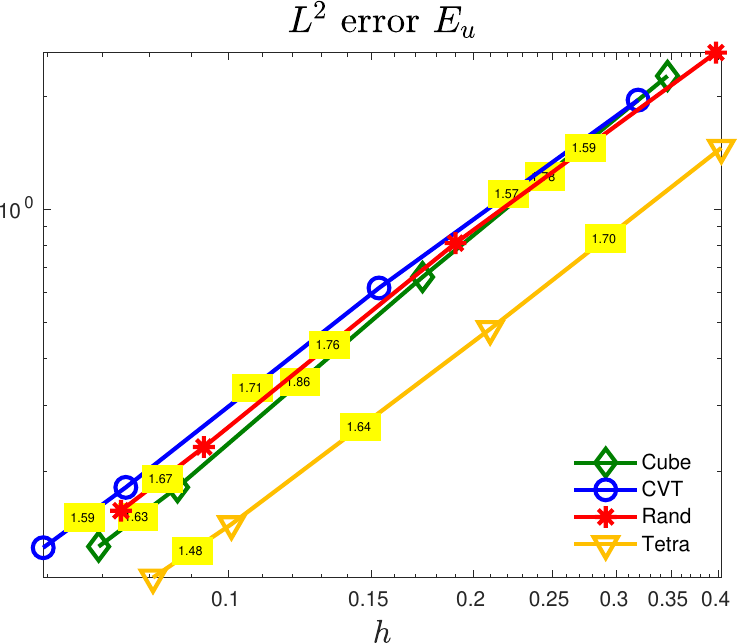}}\quad
\subfloat{\includegraphics[width=0.4\textwidth]{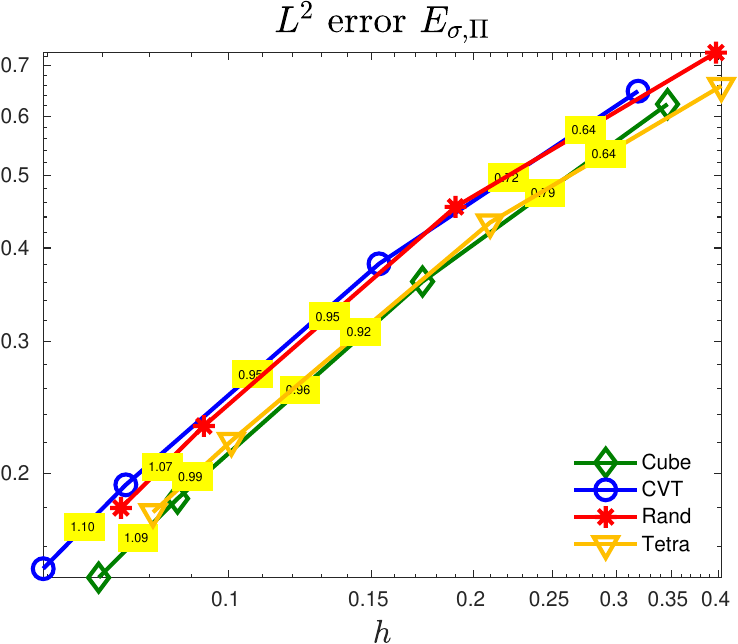}}\\
\subfloat{\includegraphics[width=0.4\textwidth]{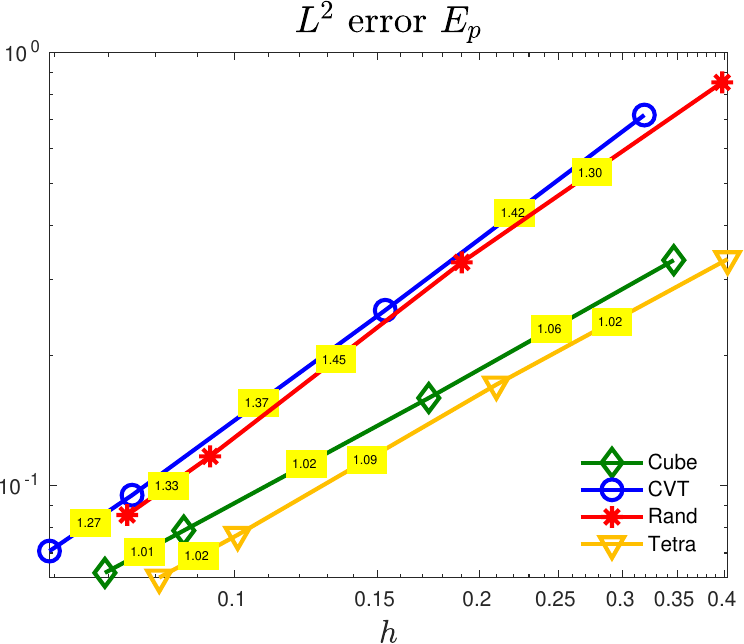}}\quad
\subfloat{\includegraphics[width=0.4\textwidth]{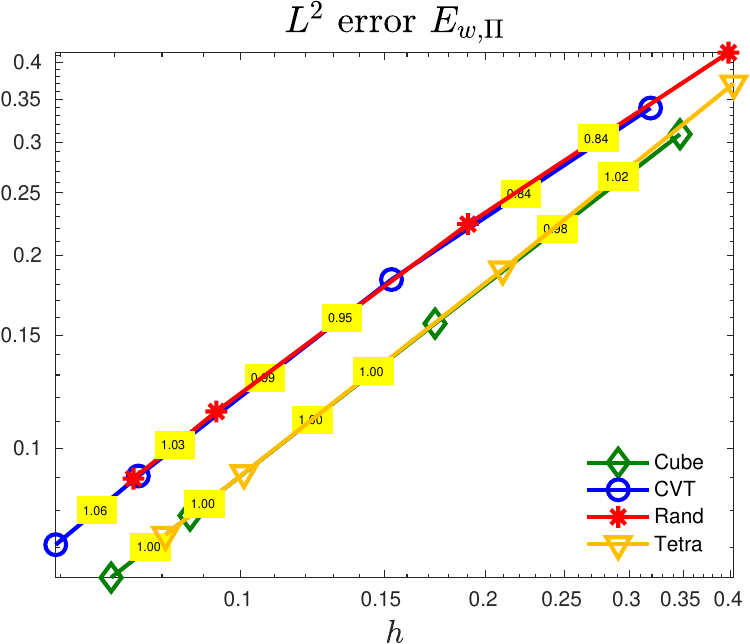}}
\caption{Test~2: a nearly-incompressible material with null storage coefficient.}~\label{fig:test2}
\end{figure}

% -> Tests
In Figure~\ref{fig:test2} we report the convergence results for the proposed VEM approach in the limit setting ($\lambda \gg 0, s_0=0)$. As expected, for the considered method, the asymptotic convergence rate is approximately equal to~1 for all error norms 
and meshes and the results are quite similar to the previous test, especially for the errors $E_{\bfu}$ and $E_p$, confirming the robustness of the method for a wide range of material parameters.

%%%%%%%%%%%%%%%%%%%%%%%%%%%%%%
\subsection{A footing problem} \label{subsection:footing-step}
%%%%%%%%%%%%%%%%%%%%%%%%%%%%%%
%-> Description
In this example, we examine a 3D footing problem, proposed in~\cite{Gaspar-Lisbona-Oosterlee-Vabishchevich_2007,Botti-DiPietro-LeMaitre-Sochala_2019, Murad-Loula_1994} for the two-dimensional case. We consider the unit cube domain $\Omega=[0,1]^3$, assuming that it is free to drain on all faces, with its bottom and lateral faces fixed. We denote $\Gamma_1=\left\{ \bfx:=(x_1,x_2,x_3) \in \partial \Omega : x_3 =1\right\}$ as the upper boundary face and apply a uniform load force on its central portion, $\Gamma_{2}:=\left\{ \bfx \in\Gamma_1 \  : \  (x_1,x_2) \in [0.3,0.7]^2 \right\}$ simulating a footing step compressing the medium, see Figure~\ref{fig:SchemaFooting}.

\begin{figure}[!ht] 
    \centering
        \includegraphics[width=0.40\textwidth]{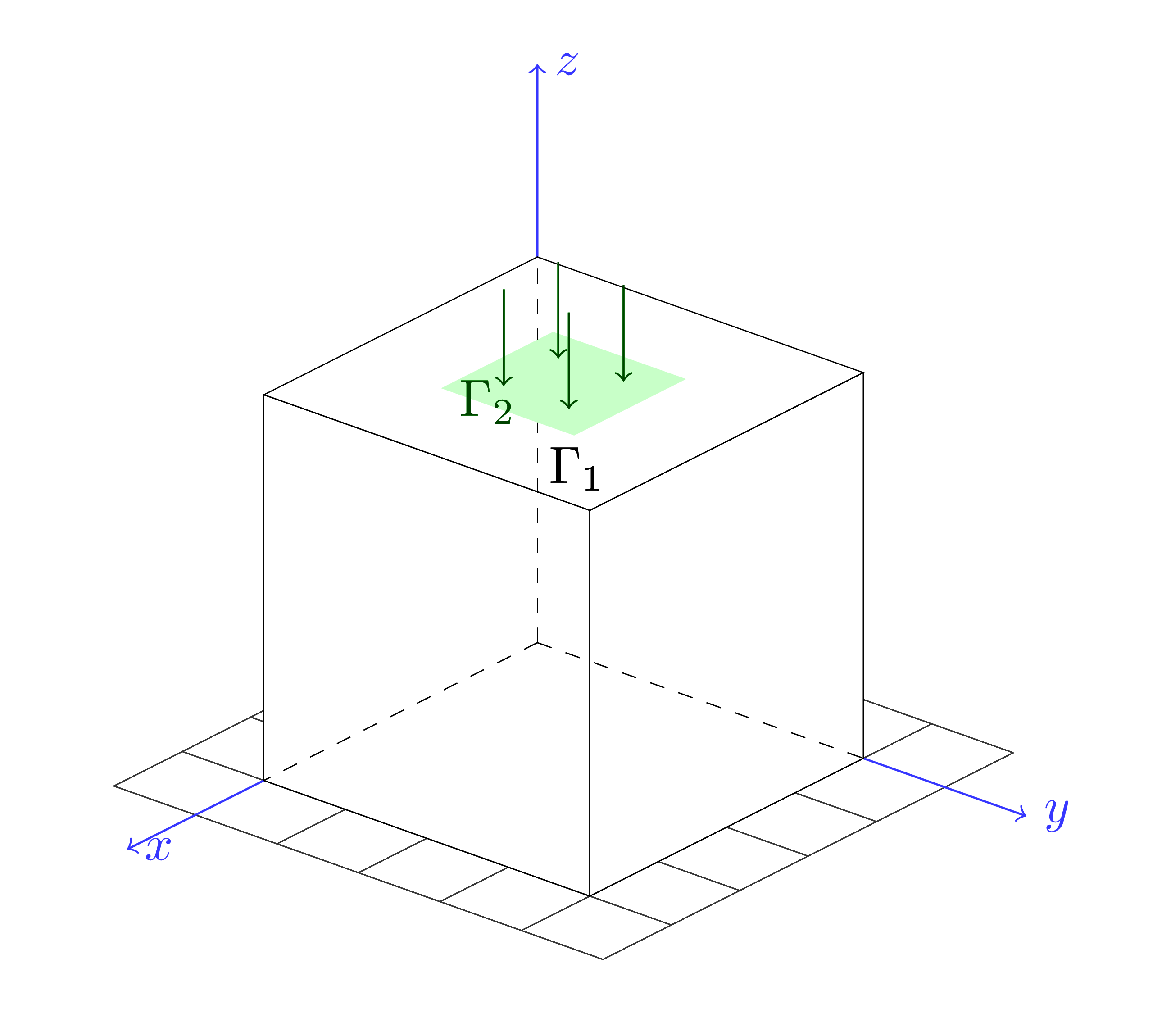}
   	\caption{Scheme for poroelasticity footing: $\Gamma_1$ is the upper boundary face and $\Gamma_2$ (the green area) is the area where we apply the uniform load force (green arrows).}\label{fig:SchemaFooting}
\end{figure}
In this case, we take the loading term $\bff$ and the initial conditions equal to zero so the solution is determined only by the following boundary conditions:
\begin{equation*}
	\begin{cases}
		\bfsigma \bfn  = (0,0,-5\, \text{kPa}), & \text{ on } \Gamma_{2}:=\left\{ \bfx \in\Gamma_1 \  : \  (x_1,x_2) \in [0.3,0.7]^2 \right\} ,\\
		\bfsigma \bfn  = \bfzero, & \text{ on } \Gamma_1 \setminus \Gamma_{2},\\
		\bfu = \bfzero, & \text{ on }\partial \Omega \setminus \Gamma_1,\\
		p= 0, & \text{ on } \partial \Omega.
	\end{cases}
\end{equation*}
We discretize the domain using a tetrahedral mesh with 40694 elements. Here, we focus our attention on the pressure profile at early times, considering $\Delta t=0.01$ and $t_f=0.2s$. The material properties of the porous medium are given in Table~\ref{table:material_properties}.
\begin{table}
	\begin{center}
		\begin{tabular}{p{5cm} p{1.5cm}p{1.7cm}}
			\hline 
			Property & Value & Unit\\
			\hline
			{Young’s modulus} &  {$3\times10^4$} &{$\text{Pa}$}\\
			{Poisson’s ratio} &  {0.2} &{ }\\
            {Permeability tensor $\bfK$}  &  {$10^{-4}$ } &{$\text{m}^2/\text{(Pa\ s)}$}\\
            {Biot--Willis coefficient $\alpha$} &  {1} &{}\\
			{Storage coefficient $s_0$} &  {0.002} &{$1/\text{Pa}$}\\
			\hline
		\end{tabular}
		\caption{Material parameters}\label{table:material_properties}
	\end{center}
\end{table}

%-> Figures and table
\begin{figure}[!ht]
    \centering
        \includegraphics[width=0.3\textwidth]{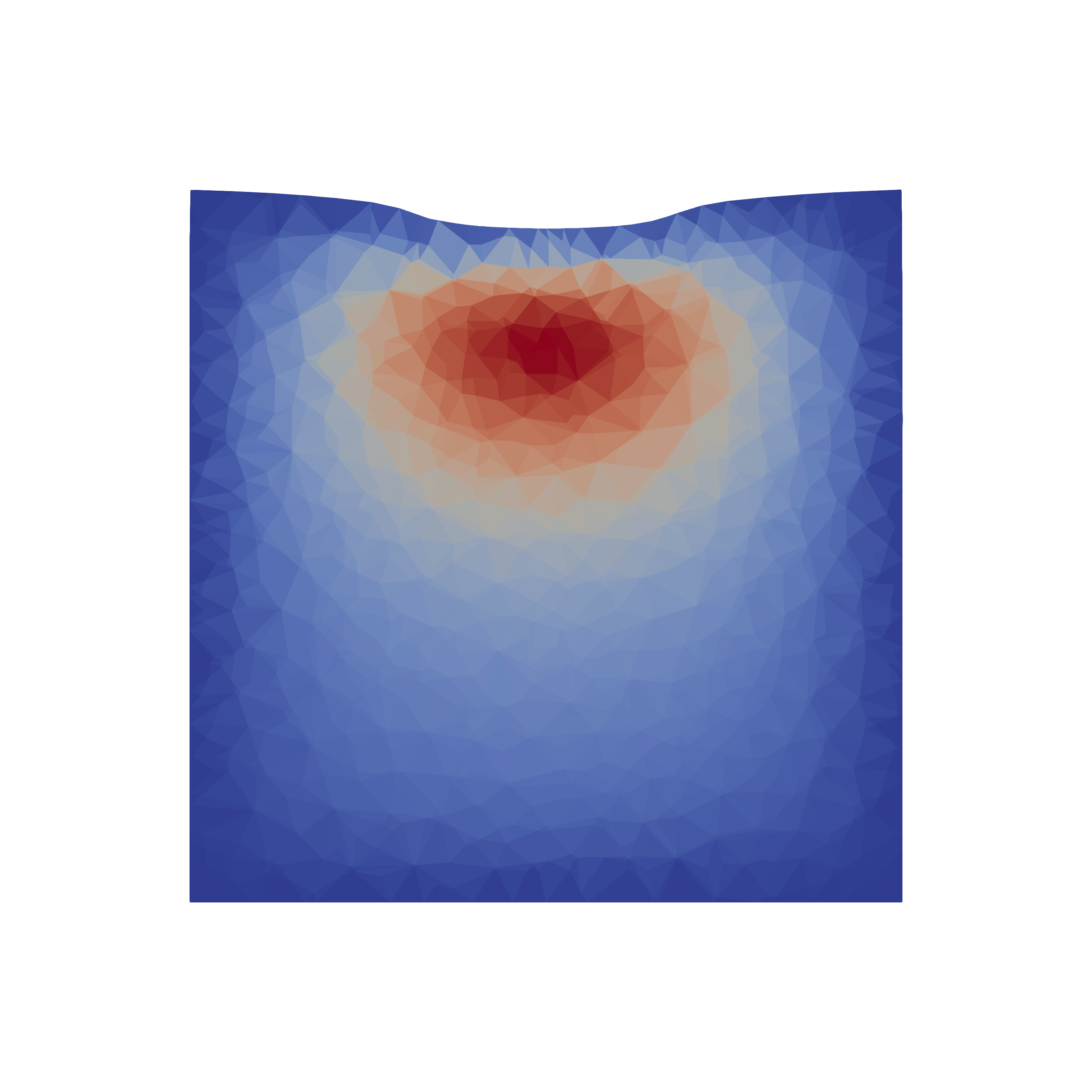}
        \includegraphics[width=0.3\textwidth]{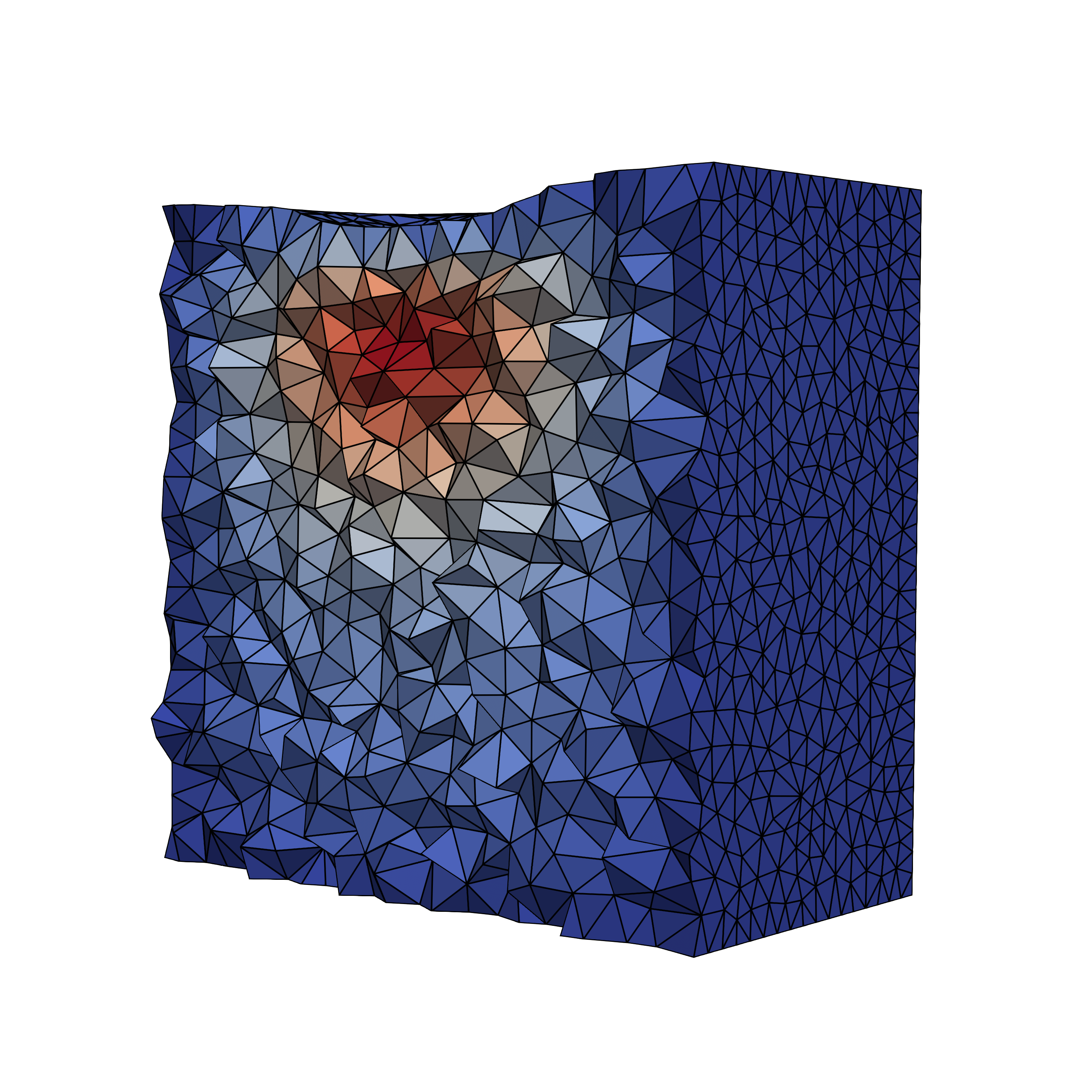}
        \includegraphics[width=0.14\textwidth]{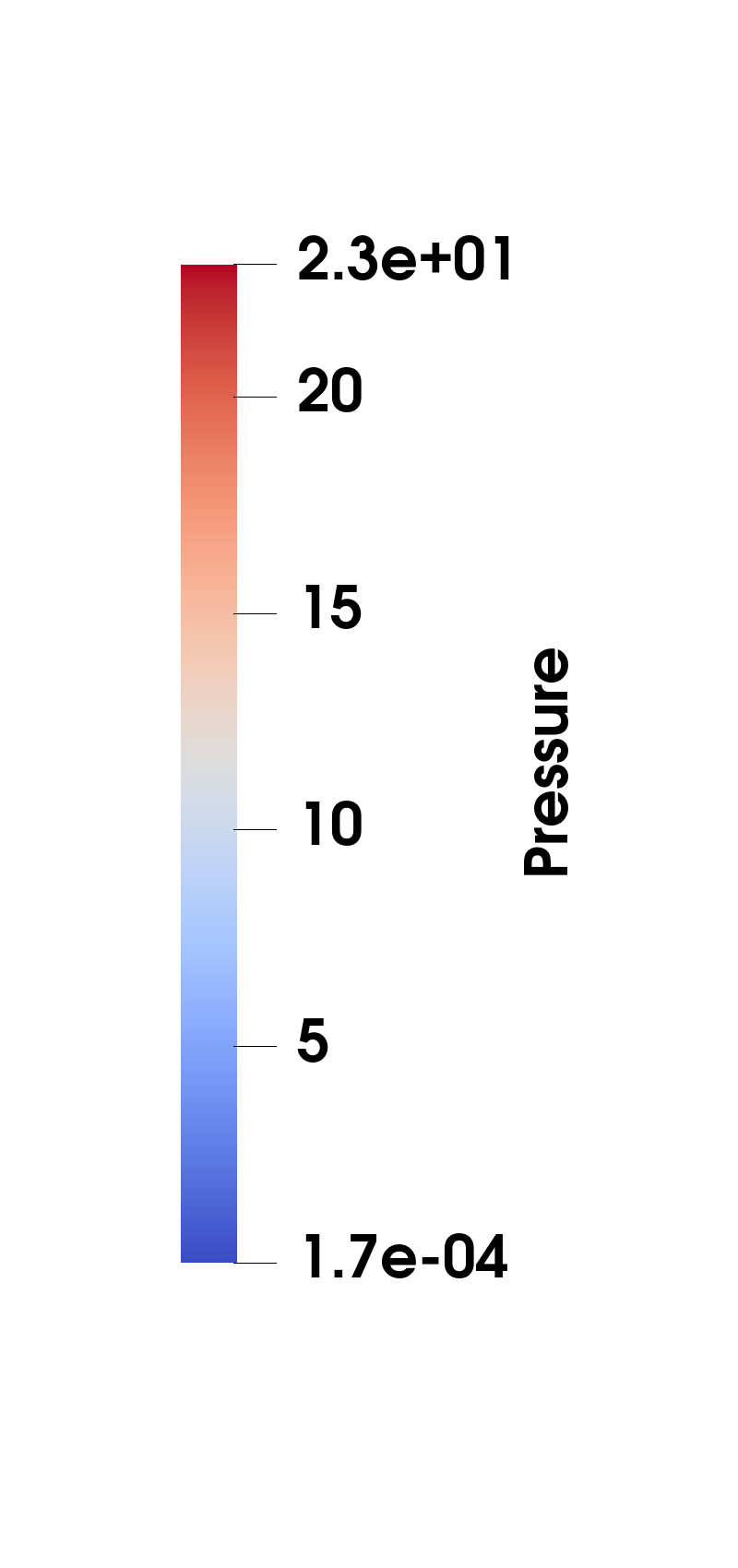}
    \caption{Pressure field (expressed in Pa) on the deformed domain at $t=0.2s$.}  \label{fig:footingProblem3D}
\end{figure}

%-> Comment
In Figure~\ref{fig:footingProblem3D}, we report the pressure field over the deformed domain. We recall that the pressure is element-wise constant, while the displacement is a rigid body motion element-by-element. The deformed mesh is constructed by moving the vertices of the mesh according to the mean value of the displacement field of the elements that share these vertices. As expected the porous medium deforms under the effect of the boundary traction and an overpressure is observed below the area subject to it.

%%%%%%%%%%%%%%%%%%%%%%%%%%%%%
\section{Conclusions}
In this work, we have proposed a lowest-order virtual element method for the fully mixed formulation of Biot's problem. The flexibility of the VEM has allowed us to enforce the symmetry of the stress tensor directly in the discrete space without resorting to techniques such as Lagrange multipliers and therefore without further increasing the method's complexity. At the same time, the VEM allows us to employ general polytopal meshes, which could be particularly useful, e.g., in the discretization of the subsurface, in particular close to intersecting fractures, faults, and thin sedimentary layers that become constraints in the meshing process.

The theoretical results in terms of convergence and stability of the method have been confirmed by numerical tests: first order convergence is observed asymptotically for a variety of mesh types, including general Voronoi meshes with small faces and short edges. Moreover, these results do not degrade in the challenging case of an incompressible material and fluid.   
To fully exploit the geometrical flexibility of the VEM, a possible future development could be to consider the inclusion of fractures or faults in the poromechanical model: beyond introducing geometrical constraints, these features can be modeled as lower dimensional domains for the fluid problem, resulting in a hybrid-dimensional model for poroelasticity in the fractures and surrounding bulk domain; moreover, frictional contact on fracture faces can be considered, leading to a nonlinear mechanical problem.
%%%%%%%%%%%%%%%%%%%%%%%%%%%%%

%%%%%%%%%%%%%%%%%%%%%%%%%%%%%
\paragraph*{Acknowledgments}
AS and MV have been partially funded by MUR (PRIN2022 research grant n. 2022MBY5JM).
MB, DP, and MV kindly acknowledge partial financial support by INdAM-GNCS project 2024 CUP\_E53C23001670001. All the authors are also members of the Gruppo Nazionale Calcolo Scientifico-Istituto Nazionale di Alta Matematica (GNCS-INdAM).
%%%%%%%%%%%%%%%%%%%%%%%%%%%%%
	
\bibliographystyle{plain} 
\bibliography{bibliography}
\end{document}